\documentclass[twoside]{article}
\usepackage[accepted]{aistats2025}
\usepackage[symbol]{footmisc}
\usepackage{hyperref}       % hyperlinks
\usepackage{url}            % simple URL typesetting
\usepackage{booktabs}       % professional-quality tables
\usepackage{amsfonts}       % blackboard math symbols
\usepackage{nicefrac}       % compact symbols for 1/2, etc.
\usepackage{microtype}      % microtypography
\usepackage{xcolor}         % colors
\usepackage{graphicx} % Required for inserting images
\usepackage{braket}
\usepackage{algorithm}
\usepackage{algorithmic}
\usepackage[round,semicolon]{natbib}
\usepackage{array}
\usepackage{titlesec}
\usepackage{amsmath,amsthm,amssymb}
\usepackage{bm}
\usepackage{verbatim}

\newcommand{\bigO}{\ensuremath{\mathop{}\mathopen{}\mathcal{O}\mathopen{}}}
\newcommand{\smallO}{ \scalebox{0.7}{$\mathcal{O}$}}

\newcommand{\EE}{{\mathbb{E}}}

\newcommand{\rank}{{\mathrm{rank}}}

\newtheorem{lemma}{Lemma}

\newtheorem{theorem}{Theorem}

\newtheorem{corollary}{Corollary}
\newtheorem{assumption}{Assumption}

\begin{document}



\twocolumn[

\aistatstitle{Noisy Low-Rank Matrix Completion via Transformed $L_1$ Regularization and its Theoretical Properties}

\aistatsauthor{ Kun Zhao\textsuperscript{1} \And Jiayi Wang\textsuperscript{1} \And  Yifei Lou\textsuperscript{2} }

\aistatsaddress{ \textsuperscript{1}The University of Texas at Dallas \And \textsuperscript{2}The University of North Carolina at Chapel Hill } ]

%\footnotetext[1]{Corresponding author, email: Jiayi.Wang2@UTDallas.edu}

\begin{abstract}
    This paper focuses on recovering an underlying matrix from its noisy partial entries, a problem commonly known as matrix completion. 
    We delve into the investigation of a non-convex regularization, referred to as transformed $L_1$ (TL1), which interpolates between the rank and the nuclear norm of matrices through a hyper-parameter $a \in (0, \infty)$.
    While some literature adopts such regularization for matrix completion, it primarily addresses scenarios with uniformly missing entries and focuses on algorithmic advances. To fill in the gap in the current literature, we provide a comprehensive statistical analysis for the estimator from a TL1-regularized recovery model under general sampling distribution. 
    In particular, we show that when $a$ is sufficiently large, the matrix recovered by the TL1-based model enjoys a convergence rate measured by the Frobenius norm, comparable to that of the model based on the nuclear norm, despite the challenges posed by the non-convexity of the TL1 regularization.
    When $a$ is small enough, we show that the rank of the estimated matrix remains a constant order when the true matrix is exactly low-rank.
    A trade-off between controlling the error and the rank is established through different choices of tuning parameters.    
    The appealing practical performance of TL1 regularization is demonstrated through a simulation study that encompasses various sampling mechanisms, as well as two real-world applications. Additionally, the role of the hyper-parameter $a$ on the TL1-based model is explored via experiments to offer guidance in practical scenarios.
\end{abstract}

\section{INTRODUCTION}
In the era of big data, low-rank matrix completion has become an indispensable and prevalent resource in various fields, such as machine learning and statistics. It addresses the challenge of estimating the missing entries of a partially contaminated observed matrix, where the low-rank property plays a pivotal role in avoiding ill-posedness. This technique is widely used for tasks like recommendation systems \citep{kang2016top,gurini2018temporal,chen2021kernel,zhang2021artificial}, signal processing \citep{weng2012low,zhang2020low,chen2023high,yuchi2023bayesian}, image recovery \citep{changjun2012single,cao2014image,zheng2024scale}, computer vision \citep{ji2010robust,jia2022non}, seismic imaging \citep{kumlu2021ground,popa2021improved,popa2022tensor}, and data imputation \citep{chen2017ensemble,chen2020nonconvex,xu2023hrst}.

To recover a low-rank matrix, matrix factorization, which reformulates the problem into a non-convex optimization task, has been extensively studied. For instance, alternating minimization methods for low-rank matrix decomposition have been explored by \cite{jain2013low} and \cite{gu2023low}. To address challenges with ill-conditioned low-rank matrices, \cite{tong2021accelerating} proposed a scaled gradient descent approach, offering improved estimation efficiency. Additionally, \cite{ahn2021riemannian} examined matrix factorization techniques through the lens of Riemannian geometry, while \cite{chen2022nonconvex} introduced a non-convex framework for matrix completion with linearly parameterized factors, further enriching the landscape of low-rank matrix recovery methods.

Despite these numerical advances, matrix factorization remains inherently limited. First, its non-convex nature means there are no general guarantees of finding a global solution, making these methods particularly vulnerable to issues such as ill-conditioning, sensitivity to initialization, and challenges in selecting the appropriate rank \citep{keshavan2010matrix, jain2013low}. Second, while low-rank matrix recovery can be formulated as a rank minimization problem, this approach is unfortunately NP-hard
%one can pose it as a rank minimization problem, which 
\citep{natarajan1995sparse}. As a result, existing algorithms for its exact solution are impractical due to the infeasibly large demands on time and computational resources \citep{chistov1984complexity}. 
To address these challenges, a popular alternative is the convex relaxation of rank, referred to as the nuclear norm, which is defined to be the sum of the singular values of the matrix. A large body of literature  \citep{recht2010guaranteed,candes2010matrix,candes2010power,koltchinskii2011nuclear,recht2011simpler,gross2011recovering,candes2012exact,klopp,klopp2015adaptive} considers the nuclear-norm regularization for low-rank matrix completion and provides theoretical guarantees for its effectiveness. %this regularizer. 

However, much of the existing work operates under the assumption of a uniform missing structure, %most works focus on the uniform sampling mechanism, 
where every entry in the matrix is assumed to be observed with equal probability \citep{candes2010matrix,candes2010power,koltchinskii2011nuclear,candes2012exact,hastie2015matrix,bhaskar2016probabilistic,cherapanamjeri2017nearly,bi2017group,chen2020noisy,xia2021statistical,farias2022uncertainty}.   
In reality, this sampling assumption is often impractical for real-world applications. For instance, regarding the well-known Netflix Prize \citep{bennett2007netflix}, the dataset can be represented by a low-rank matrix with users as rows, movies as columns, and ratings as entry values. Certain movies are rated more frequently than others, and some users rate more movies than others, leading to a non-uniform distribution of known entries.
Consequently, there has been increasing attention on non-uniform sampling, with related works generally falling into two categories: missing at random (MAR) \citep{srebro2010collaborative,kiraly2015algebraic,chen2015completing,cho2017asymptotic} and missing not at random (MNAR) \citep{ma2019missing,sportisse2020imputation, jin2022matrix, lipairwise}. Since the case of MNAR is known to be particularly challenging due to the identification  issue \citep{little2019statistical}, 
%we plan to explore this in future research endeavors. 
we primarily focus on the MAR scenario following the work \citep{macdonald2002missing}.

 Under the MAR setting, \cite{klopp} applied the standard nuclear norm penalty to deal with general sampling distributions and provided theoretical insights on the estimation errors measured in the Frobenius norm. 
However, 
%\cite{fang2018max} demonstrated %\yl{[theoretically or numerically or empirically? basically how they demonstrate this]} that the nuclear norm performs poorly when sampling schemes are not uniform. 
it is reported empirically that the nuclear norm overestimates the matrix's rank \citep{wang2021matrix}.
Alternatively, the max-norm was first proposed by \cite{srebro2004maximum} for matrix completion under non-uniform sampling mechanisms. Later, \cite{cai2016matrix} proved its theoretical superiority over the nuclear norm for noisy matrix completion under a general sampling model. Sequentially,  \cite{fang2018max} proposed a more flexible estimator, called the hybrid regularizer, which incorporates both max-norm and nuclear norm. They demonstrated the performance of this hybrid regularizer in comparison to the max-norm and nuclear norm under various settings. However, the hybrid approach only achieves a sub-optimal rate for recovering exact low-rank matrices under a uniform sampling scheme, compared to the nuclear norm, and it comes at the cost of a high computational burden due to the use of semidefinite programming (SDP) \citep{srebro2004maximum}.

Another promising approach that balances computational efficiency and recovery performance is the transformed $L_1$ (TL1) regularization, originating in sparse recovery  \citep{zhang2018minimization,ma2019transformed}. When applied to a vector, TL1 can interpolate between the $L_0$ semi-norm and the $L_1$ norm with a hyper-parameter $a \in (0, \infty)$. Consequently, when applied to the vector composed of the singular values of a matrix, TL1 can approximate both the rank and the nuclear norm of the matrix \citep{zhang2015transformed}. This non-convex regularization offers a closer approximation to the rank function than the convex nuclear norm, better capturing the low-rank structure of the matrix and retaining the computational complexity of nuclear norm minimization. %While 
\cite{zhang2015transformed} provided the convergence analysis of the numerical algorithm under a uniform sampling scheme. 
However, to the best of our knowledge, TL1 regularization has been empirically explored for low-rank matrix completion under uniform sampling, while its theoretical guarantees and recovery performance under more general sampling schemes have yet to be thoroughly examined.

In this paper, we provide a comprehensive statistical analysis of the estimator derived from a TL1-regularized recovery model under a general sampling distribution. Particularly, we demonstrate that for a sufficiently large $a$, the estimator of the target matrix using TL1 regularization achieves the optimal convergence rate on the Frobenius norm error, comparable to that of the nuclear norm-based model. This theoretical guarantee aligns with the property of the TL1 function: as $a\rightarrow \infty$, the TL1 function approximates the nuclear norm more closely. 
When $a$ is small enough, we establish a sub-optimal convergence rate for the estimation error and an upper bound on the rank of the estimator. 
Our results imply that by choosing a smaller value of $a$, the estimated matrix tends to have a lower rank, which can be advantageous in applications where a low-rank solution is preferred for reasons such as simplicity or interpretability. This phenomenon is in line with another property of the TL1 function: as $a\rightarrow 0$, the TL1 function behaves more like the rank function. In this sense, TL1 serves as a valuable theoretical tool to enhance our understanding of the performance between rank minimization and nuclear norm minimization.
%We highlight that obtaining such an upper bound \yl{[we didn't talk about upper bound before this sentence]} on the estimation error is challenging due to the non-convex nature of the TL1 function and the absence of the traditional triangle inequality for a norm, unlike the typical analysis of the convex nuclear norm.

Moreover, TL1 is a practical tool, offering significant improvements in accuracy over other methods, regardless of whether it is under uniform or non-uniform sampling with noisy data. The effectiveness of TL1 regularization is demonstrated through a comprehensive simulation study under various missing data mechanisms, highlighting TL1 regularization's adaptability and robustness. %Specifically, our experimental results show that 
% the nuclear norm works well for uniform sampling, the max norm well for non-uniform sampling, while the TL1 regularization is the best for any sampling scheme.
We also validate its performance using two real-world data applications, providing empirical evidence of its superiority in practical scenarios. Furthermore, the role of the hyper-parameter $a$ is examined empirically through numerical experiments, which reveals the trade-off between error and rank estimation. 

To the best of our knowledge, this work is the first to provide a statistically theoretical analysis of error bounds for TL1 regularization in matrix completion. In fact, it is the pioneering theoretical analysis of a nonconvex method for matrix completion, shedding light on the analysis of other nonconvex regularizations for matrix completion.
%\jiayi{Maybe we can write something like our work can shed light on the analysis of other type of nonconvex regularizations for matrix completion such as SCAD and MCP...} 
Additionally, we want to emphasize that it is technically challenging to establish our theorems, as some techniques and properties associated with convex approaches, such as the nuclear norm \citep{candes2010matrix,klopp} and max-norm \citep{cai2016matrix}, are not directly applicable to the nonconvex TL1 method. And yet, our results are on par with those in convex scenarios. Specifically, Theorems \ref{Them1} and \ref{Them2} are minimax optimal (up to a logarithmic factor) in terms of the order with the current literature. Empirically, we demonstrate through extensive experiments that TL1 generally achieves significant improvements over its convex counterparts (including the nuclear norm, max-norm, and a hybrid approach) across various noise levels and sampling distributions.

\section{PRELIMINARIES}

\subsection{Problem Setup}

We aim to reconstruct an underlying matrix $A_0 \in \mathbb{R}^{m_1 \times m_2}$ from its partial entries coded by a set of matrices $T_i \in \mathbb{R}^{m_1 \times m_2}$, $i =1,\dots, n$, which are i.i.d.~copies of a random indicator matrix with distribution $\Pi=(\pi_{kl})_{k,l=1}^{m_1, m_2}$ over the set:
\begin{equation}
\setlength{\abovedisplayskip}{2pt}
\setlength{\belowdisplayskip}{2pt}
    \Gamma=\{e_k(m_1)e_l^\top(m_2), k\in[m_1], l\in [m_2]\},
\end{equation}

where $\pi_{kl}$ is the probability that a particular sample is at the location $(k,l)$, $e_k(m_j)$ represents the canonical basis vector in $\mathbb{R}^{m_j}$ whose $k$-th entry is 1 and other entries are zeroes, $[m_j]=\{1,\dots,m_j\}$ for $j=1,2,$ and $n$ denotes the number of observed samples. 
In other word, we observe $n$  independent pairs $(T_i, Y_i)$ with $T_i\in \Gamma$ and $Y_i\in \mathbb R$ that adhere to 
the trace regression model \citep{negahban2011estimation}
\begin{equation} \label{trace_reg_mod}
\setlength{\abovedisplayskip}{2pt}
\setlength{\belowdisplayskip}{2pt}
    Y_i=\text{tr}(T_i^\intercal A_0) + \sigma\xi_i = \langle T_i, A_0 \rangle + \sigma\xi_i, %\quad i=1,\dots,n,
\end{equation}
where $\sigma >0$ denotes the standard deviation and $\xi_i$ are independent noise variables with $\mathbb{E}(\xi_i) =0$ and $\mathbb{E}(\xi_i^2) = 1$. 
Our goal is to estimate the matrix $A_0$ from observations $Y_i$ and $T_i$, $i=1,\dots,n$.

\subsection{Notations}

We introduce the notations to be used throughout this paper. For a matrix $A\in\mathbb{R}^{m_1 \times m_2}$, we define constants $m = \min\{m_1, m_2\} = (m_1 \wedge m_2)$, $M = \max\{m_1,m_2\} = (m_1 \vee m_2),$ and $d = m_1 + m_2$. We denote the trace of the matrix $A$ by $\text{tr}(A)$.
We define two standard matrix norms, i.e.,
\begin{equation*}
\|A\|_\infty = \underset{k,l}{\text{max}} |A(k, l)|  \ \mbox{and} \ \|A\|_F = \sqrt{\underset{k,l}{\sum}A^2(k, l)}, 
\end{equation*}
where $A(k, l)$ denotes the value of $(k,l)$-th entry of $A$. Denote $\sigma_j(A)$ as the $j$th singular values of $A$  in decreasing order, then the nuclear norm is defined as $\|A\|_* = \sum_{j=1}^{m} \sigma_j(A)$ and the spectral norm $\|A\| = \sigma_1(A)$. Given the sampling distribution $\Pi$, we define 
$L_2(\Pi)$ norm of $A$  by 
%\begin{equation}
    $\|A\|^2_{L_2(\Pi)}=\mathbb{E}(\langle A, T \rangle^2)=\sum_{k=1}^{m_1}\sum_{l=1}^{m_2} \pi_{kl}A^2(k,l).$
%\end{equation}

Lastly, our analysis requires the following asymptotic notations. For two non-negative sequences $\{a_n\}$ and $\{b_n\}$, we say $a_n = \mathcal{O}(b_n)$ if there exists a constant $C$ such that  $a_n \leq Cb_n$ and $a_n = \mathcal{O}_p(b_n)$ if there exists a constant $C'$ such that  $a_n \leq C'b_n$ with high probability; $a_n = \smallO(b_n)$ if there is a constant $C''$ such that $a_n < C''b_n$. We denote $a_n \asymp b_n$ if $a_n = \mathcal{O}(b_n)$ and $b_n = \mathcal{O}(a_n)$.

\subsection{Model Formulation}
\label{sec:problem}
TL1 regularization has been used to promote the sparsity in signal recovery \citep{zhang2018minimization}, which can be applied to the singular values of a matrix for low rankness \citep{zhang2015transformed}. Specifically, the TL1 on the matrix $A$ is defined by
\begin{equation}\label{eq:TL1}
\setlength{\abovedisplayskip}{2pt}
\setlength{\belowdisplayskip}{2pt}
    \text{TL1}_a(A) = \sum_{j=1}^{m} \frac{(a+1)\sigma_j(A)}{a+\sigma_j(A)}, %\quad %\text{where parameter} \; .
\end{equation}
with an internal parameter $a \in (0,\; \infty)$.
The behavior of the TL1 function varies significantly with the parameter $a$. Specifically, the function has two useful limits:
\setlength{\abovedisplayskip}{4pt}
\setlength{\belowdisplayskip}{4pt}
\begin{align}\label{limits}
    \underset{a \rightarrow 0+}{\lim} \text{TL1}_a(A) = \rank(A), \; \underset{a \rightarrow \infty}{\lim} \text{TL1}_a(A) = \|A\|_*,
\end{align}
allowing the TL1 penalty to act as a bridge between the $L_0$ semi-norm, which counts non-zero singular values to measure matrix rank directly, and the nuclear norm, which sums singular values as a convex relaxation of the rank. Unlike the $L_0$ semi-norm, the TL1 penalty with $a>0$ is continuous everywhere, which is beneficial for optimization.  Since TL1 imposes a heavier penalty on smaller singular values, pushing them toward zero more aggressively than the nuclear norm does, it can induce lower-rank solutions than the nuclear norm, particularly at smaller values of $a$. However, its non-convex nature poses challenges because non-convex functions typically have multiple local minima and possibly saddle points, complicating the search for a global minimum.

To estimate the target matrix $A_0$, we incorporate the TL1 function into a least-squares fit of the trace regression, thus leading to 
\setlength{\abovedisplayskip}{3pt}
\setlength{\belowdisplayskip}{3pt}
\begin{align}\label{est:hatA}
    \hat{A} = \underset{\|A\|_\infty \leq \zeta}{\arg\min} \bigg\{ \frac{1}{n} \sum_{i=1}^{n} (Y_i - \langle T_i, A \rangle)^2 + \lambda \text{TL1}_a (A) \bigg\}, 
\end{align}
where $\lambda > 0$ balances between two terms and $\zeta$ is a (tunable) upper bound on the estimator.

This paper focuses on the theoretical analysis of the estimator $\hat{A}$ defined in \eqref{est:hatA}. In particular, we aim to find a non-asymptotic upper bound to quantify the difference between the estimated matrix $\hat{A}$ and the true matrix $A_0$, measured by the Frobenius norm. Our analysis requires $\|\hat A\|_\infty\leq \zeta$, while $\zeta$ can be determined in practice, e.g., $\zeta=5$ for Netflix ratings in the range of $[0,5]$.  To make the paper self-contained, we include an efficient algorithm for minimizing \eqref{est:hatA} via the alternating direction method of multipliers (ADMM) \citep{boydPCPE11admm}  in Appendix \ref{A4:Algorithm}.

\section{THEORETICAL PROPERTIES}

Define $C_l = \sum_{k=1}^{m_1} \pi_{kl}$ as the probability that an observation appears in the $l$-th column and $R_k = \sum_{l=1}^{m_2} \pi_{kl}$ as the probability that an observation appears in the $k$-th row, 
where $k\in [m_1]$ and $l\in[m_2]$. By the definitions along with the constraints $\sum_{l=1}^{m_2} C_l = 1$ and $\sum_{k=1}^{m_1} R_k = 1$, we have $\underset{l}{\max}\; C_l \geq 1/m_2$ and $\underset{k}{\max} \; R_k \geq 1/m_1$,  implying that $\underset{k,l}{\max}(R_k, C_l) \geq 1/m$. 
%$\underset{k,l}{\max}(R_k, C_l) \geq 1/m$.

Our theoretical analysis requires the following three assumptions:
 
\begin{assumption}\label{assump1}
   There exists a constant $L \geq 1$ such that the maximum of $R_k$ and $C_l$ over $k$ and $l$ has an upper bound: $ \underset{k,l}\max(R_k, C_l) \leq L/m$.
    %Note that $\underset{k,l}\max(R_k, C_l) \geq \frac{1}{m}$ is inherently satisfied.
\end{assumption}

\begin{assumption}\label{assump2}
    There exists a constant $\nu \geq 1$ such that $1/(\nu m_1 m_2) \leq \pi_{kl} \leq \nu/(m_1 m_2).$
\end{assumption}

\begin{assumption}\label{assump3}
    There exists a constant $c_0 > 0$ such that $\underset{i=1,\dots,n}\max\EE [\exp (|\xi_i|/c_0)] \leq e$, where $e$ is the base of the natural logarithm. %\yl{[what's e?]}
\end{assumption}

These assumptions are commonly used in the literature \citep{koltchinskii2011nuclear,klopp,cai2016matrix, klopp2017robust}. 
In Assumption \ref{assump1}, a larger value of $L$ indicates greater imbalance in sampling, leading to a non-uniform setting. 
%\kun{In Assumption \ref{assump1}, the value of $L$ reflects the extent of the discrepancy between $\underset{l}{\max} \; C_l$ and $\underset{l}{\min} \; C_l$, \yl{[L doesn't involve min of $C_l?$]} and a similar effect applies to $R_k$ for $k \in [m_1]$. A larger $L$ indicates a greater difference between the extremes, implying more imbalance in sampling that results in a nonuniform setting, and it could lead to a higher error between the reconstructed matrix and the actual matrix. Additionally, by the definitions of $C_l$ and $R_k$, we have $\underset{l}{\max}\; C_l \geq 1/m_1$ and $\underset{k}{\max} \; R_k \geq 1/m_2$, which implies that $\underset{k,l}{\max}(R_k, C_l) \geq 1/m$. In our analysis, we focus on the scenarios when they \yl{[not clear what do "they" refer to]} do not differ drastically, meaning no single column or row is sampled with excessively high probability.} 
Assumption \ref{assump2} implies %\yl{[assumption establishes is weird? why not using implies?]} 
that every matrix entry has a non-zero probability of being observed and we have $1/(\nu m_1m_2) \|A\|_F^2 \leq \|A\|_{L_2(\Pi)}^2 \leq \nu/(m_1m_2) \|A\|_F^2$, 
where $\nu$ is a constant independent of $n$ or the matrix dimensions.

%Assumptions \ref{assump1} and \ref{assump2} ensure that no single row or column can dominate the observed data. 
%By capping the maximum aggregate probabilities, 
Therefore, these two assumptions maintain a balance in the influence exerted by each row and column, ensuring that every matrix element has a chance of being observed, thereby avoiding issues where certain data points are never seen.
For a uniform distribution of observed elements, $L$ and $\nu$ are taken as 1.
%we can set $L = \nu = 1$.  
%\jiayi{$L$ and $\nu$ is taken as $1$.}
Assumption \ref{assump3} specifies the sub-exponential %sub-Gaussian 
distribution of the noise term, which is a mild assumption on the noise.

As revealed in Section \ref{sec:problem}, TL1 regularization interpolates between the rank and the nuclear norm of matrices depending on the value of $a$. In the following, we discuss the theoretical properties of the estimator within two regimes: Regime 1 is when $a$ is large, while Regime 2 is when $a$ is small.

\subsection{Regime 1: when \texorpdfstring{$a$}{a} is large}

We demonstrate in Theorems \ref{Them1} and \ref{Them2} that the estimator obtained by TL1 regularization in \eqref{est:hatA} achieves the same convergence rate as that obtained by nuclear norm regularization, 
%\yl{[cite the paper for the same convergence rate by nuclear norm]} 
which implies that the TL1 regularization behaves like the nuclear norm when $a$ is asymptotically large, i.e., $a^{-1} = \mathcal{O}((\zeta \sqrt{m_1m_2})^{-1})$. 
%\yl{[why not $a=O(X)$, but using $a^{-1}=O(X^{-1})$?]}.

We begin with Theorem \ref{Them1} to establish an upper bound for the estimation error when the true matrix $A_0$ is approximately low-rank in the sense that the nuclear norm of $A_0$ is properly controlled.

\begin{theorem}\label{Them1}
\vspace{-1mm}
   Suppose Assumptions \ref{assump1}-\ref{assump3} hold, $A_0 \in \mathbb{R}^{m_1 \times m_2}$ is approximately low-rank in the sense that $\|A_0\|_*/\sqrt{m_1m_2} \leq \gamma$ for a constant $\gamma > 0$, and $\|A_0\|_\infty \leq \zeta$ for a constant $\zeta$. 
   Take $\lambda \asymp \frac{(\zeta \vee \sigma)}{\sqrt{m_1m_2}} \frac{a + \zeta \sqrt{m_1m_2}}{1+a} \sqrt{\frac{Ld\log d}{n}}$, where $d = m_1+m_2$, then for any $n \gtrsim d\log d $ and  $a^{-1} = \mathcal{O}((\zeta \sqrt{m_1m_2})^{-1})$, there exist two constants $C_1$ and $C_2$ only depending on $c_0$ such that the estimator $\hat A$ from \eqref{est:hatA} satisfies
    \begin{multline}\label{ineq:thm1}
         \frac{1}{m_1m_2}\|\hat{A} - A_0\|_F^2 \leq C_1 \nu (\zeta \vee \sigma) \gamma \sqrt{\frac{Ld\log d}{n}} \\
         + C_2 \nu \zeta^2 \sqrt{\frac{L\log d}{n}},
    \end{multline}
   with probability at least $1-(\kappa+1)/d$, where $\kappa$ is a constant depending on $L$.
   \vspace{-2mm}
\end{theorem}

Please refer to Appendix \ref{A1:Theorem1,3} for the proof of Theorem \ref{Them1}. %Note that Theorem \ref{Them1} is a special instance of Lemma \ref{lemma3} in Appendix \ref{A1:Theorem1,3} that gives a general bound of the estimation error for any value of $a$. 
%\yl{[comment on what does the bound in term of big$O$ in Thm 1 help us]}

When $n \gtrsim d\log d $, the first component in \eqref{ineq:thm1} dominates in our bound
%\yl{[why dominate? mention which variable going to infinity?]} 
and it is comparable to the result for nearly low-rank matrices in \cite{negahban2012restricted}. Though the second component with order $\bigO(\sqrt{(\log d)/n})$ is slightly discrepant with a higher order term ($\bigO(n^{-1})$) compared to \cite{negahban2012restricted}, it is negligible. Additionally, our bound reaches the minimax optimal rate up to a logarithmic order under the uniform sampling  \citep{cai2016matrix}.

When $A_0$ is exactly low-rank, Theorem \ref{Them2} shows that our estimator can achieve a tighter bound than the one in Theorem \ref{Them1}. This is the same situation as the estimation error for nuclear norm regularization
\citep{koltchinskii2011nuclear,klopp}. 
However, the proof for the case of the nuclear norm \citep{negahban2012restricted,klopp} is not applicable for TL1 regularization, because, unlike the convex nuclear norm, the triangle inequality does not hold for TL1. 
Instead, we carefully analyze the gradient of the TL1 function to obtain the bound; please refer to Appendix \ref{A2:Them2} 
for the proof of Theorem \ref{Them2}.

\begin{theorem}\label{Them2}

    Suppose Assumptions \ref{assump1}-\ref{assump3} hold, 
     $A_0 \in \mathbb{R}^{m_1 \times m_2}$ is exactly low-rank, i.e., $\mathrm{rank}(A_0) \leq \tau$ for an integer $\tau$, and $\|A_0\|_\infty \leq \zeta$ for a constant $\zeta$.  
   Take $\lambda \asymp \frac{(\zeta \vee \sigma)}{\sqrt{m_1m_2}} \frac{a + \zeta \sqrt{m_1m_2}}{1+a} \sqrt{\frac{Ld\log d}{n}}$, where $d = m_1+m_2$, then for any $n \gtrsim d\log d $ and  $a^{-1} = \mathcal{O}((\zeta \sqrt{m_1m_2})^{-1})$, there exist two constants $C_3$ and $C_4$ only depending on $c_0$ such that the estimator $\hat A$ from \eqref{est:hatA} satisfies
    \begin{multline}\label{ineq:thm2}
        \frac{1}{m_1m_2}\|\hat{A} - A_0\|_F^2 \leq C_3 \nu^2(\zeta^2 \vee \sigma^2) \rank(A_0) \frac{Ld\log d}{n}\\
        + C_4 \nu \zeta^2 \sqrt{\frac{L\log d}{n}},
    \end{multline}
     with probability at least $1-(\kappa+1)/d$, where $\kappa$ is a constant depending on $L$.
\end{theorem}

The first component in \eqref{ineq:thm2} performs on par with the work \citep{negahban2012restricted}.
Furthermore, it dominates the second component if $n \leq d^2$, which is a mild condition given that $n\ll m_1m_2$ and $d=m_1+m_2$.
The upper bound we derive in \eqref{ineq:thm2} for exactly low-rank matrices achieves the optimal convergence rate in a minimax sense up to a logarithmic factor in line with the existing literature \citep{candes2010matrix,negahban2012restricted,klopp}.

\subsection{Regime 2: when \texorpdfstring{$a$}{a} is small}

When $a$ approaches 0, the TL1 function approximates the rank function, as indicated in \eqref{limits}. This behavior enhances the ability of TL1 regularization to control the rank of the estimated matrix, which is particularly useful in applications involving low-rank structures. We construct a theoretical foundation to estimate an upper bound of errors in recovering exactly low-rank matrices when $a$ is sufficiently small.
We start with a non-asymptotic error bound of exactly low-rank matrices for any value of $a$ that falls outside the range specified in Regime 1, i.e., %$a^{-1} = \mathcal{O}((\zeta \sqrt{m_1m_2})^{-1}),$ verses $a = \mathcal{O}(\zeta \sqrt{m_1m_2})$ for Regime 2.
$a = \mathcal{O}(\zeta \sqrt{m_1m_2})$ for Regime 2.

\begin{theorem}\label{Them3}
  Suppose Assumptions \ref{assump1}-\ref{assump3} hold,  $A_0 \in \mathbb{R}^{m_1 \times m_2}$ is exactly low-rank, i.e., $\mathrm{rank}(A_0) \leq \tau$ for an integer $\tau$, and $\|A_0\|_\infty \leq \zeta$ for a constant $\zeta$. Take $\lambda^{-1} = \mathcal{O} \left( \left( \frac{(\zeta \vee \sigma)}{\sqrt{m_1m_2}} \frac{a + \zeta\sqrt{m_1m_2}}{1+a} \sqrt{\frac{Ld\log d}{n}} \right)^{-1} \right)$, where $d = m_1+m_2$, then 
  %\yl{[In Thm 4, it's if a blabla; make consistent?]} 
  for any $n \gtrsim d\log d$ and $a = \mathcal{O}(\zeta \sqrt{m_1m_2})$, there exist two constants $C_5$ and $C_6$ only depending on $c_0$ such that the estimator $\hat A$ from \eqref{est:hatA} satisfies
    \begin{multline}\label{ineq:thm3}
        \frac{1}{m_1m_2}\|\hat{A} - A_0\|_F^2 \leq C_5 \lambda \nu \frac{(1+a)\zeta \sqrt{m_1m_2}}{a+\zeta \sqrt{m_1m_2}}\rank(A_0)\\  
        + C_6 \nu \zeta^2\sqrt{\frac{L\log d}{n}},
    \end{multline}    
with probability at least $1-(\kappa+1)/d$, where $\kappa$ is a constant depending on $L$.
\end{theorem}

%Theorem \ref{Them3} is a special case of Lemma \kun{3} %\ref{lemma3} 
Please refer to Appendix \ref{A1:Theorem1,3} for the proof.
%when $A_0$ is exactly low-rank. 
 %indicating that larger values of $\lambda$ cause larger errors.

Note that even if $\lambda$ is chosen to be of the exact order in that condition, the error bound in \eqref{ineq:thm3} is not as tight as the one of Theorem \ref{Them2}. A specific %\yl{[sharper?]} 
error bound is presented in Corollary \ref{corollary1} with appropriate choices of $\lambda$ and $a$. 
It is unclear whether this upper bound can be further sharpened, which will be left as a direction for future research.
However, by selecting a smaller $a$, we can gain some control over the rank of the estimator in Theorem \ref{Them4}.

\begin{theorem}[Low-rankness]\label{Them4}
  Suppose Assumptions \ref{assump1}-\ref{assump3} hold, $A_0 \in \mathbb{R}^{m_1 \times m_2}$ is exactly low-rank, i.e., $\mathrm{rank}(A_0) \leq \tau$ for an integer $\tau$, and $\|A_0\|_\infty \leq \zeta$ for a constant $\zeta$. 
  Take $\lambda^{-1} = \mathcal{O} \left( \left( \frac{(\zeta \vee \sigma)}{\sqrt{m_1m_2}} \frac{a + \zeta\sqrt{m_1m_2}}{1+a} \sqrt{\frac{Ld\log d}{n}} \right)^{-1} \right)$, where $d = m_1+m_2$,   
  then for any $n \gtrsim d\log d$ and $a = \smallO((m_1m_2)^{1/4})$, there exists a constant $C_7$ only depending on $c_0$ such that the estimator $\hat A$ from \eqref{est:hatA} satisfies
    \begin{align}\label{eqn:rank}
    %\footnotesize
        \rank (\hat{A}) \leq  \, &C_7  \Bigg\{ \lambda^{-1} \frac{Ld \log d\sqrt{m_1m_2}}{(1+a)(a+ \sqrt{m_1m_2})n} \notag\\
        &\times \left({\sqrt{a}}/{\left( a + \sqrt{m_1 m_2} \right)^{1/4}} + 1 \right)^2  \\
       & + \rank(A_0) \left({\sqrt{a}}/{\left( a + \sqrt{m_1 m_2} \right)^{1/4}} + 1 \right) \Bigg\} \notag, 
    \end{align}
    with high probability. 
\end{theorem}

Theorem \ref{Them4} provides a theoretical control on the rank of the matrix estimated by the TL1 regularization, showing the estimated rank decreases as $\lambda$ increases for sufficiently small $a$.
Theorem \ref{Them4} indicates that TL1 regularization controls the rank of estimated matrices to some extent and 
implies that TL1 regularization effectively promotes sparsity in the singular values. 
To the best of our knowledge, there is no comprehensive analysis in the current literature that bounds the rank of estimated matrices by nuclear norm or max-norm.

Furthermore, there is a trade-off between controlling the estimation error and the rank. For a fixed value of $a$,  a larger value of $\lambda$ makes the first component in the bound \eqref{eqn:rank} smaller, thus resulting in a lower rank estimation. 
However, a lower rank estimator comes at the cost of a slower convergence rate of the errors in the Frobenius norm, compared to the rate specified in Theorem \ref{Them3}. This trade-off underlines the balance between achieving a low-rank representation and maintaining a high convergence rate (or minimizing error estimation). In addition, smaller $a$ makes the second component in \eqref{eqn:rank} smaller and yields a lower rank estimation, highlighting the importance of carefully choosing the hyper-parameters $\lambda$ and $a$ to optimize both rank reduction and error minimization in a balanced manner.

Combining Theorems \ref{Them3} and \ref{Them4}, we present a specific upper bound for the estimation error and the rank with appropriate choices of $\lambda$ and $a$ in Corollary \ref{corollary1}. The proofs of Theorem \ref{Them4} and Corollary \ref{corollary1} can be found in Appendix \ref{A3:Them4,Cor1}.

\begin{corollary}\label{corollary1}
  Suppose Assumptions \ref{assump1}-\ref{assump3} hold, $A_0 \in \mathbb{R}^{m_1 \times m_2}$ is exactly low-rank, i.e., $\mathrm{rank}(A_0) \leq \tau$ for an integer $\tau$. Take $\lambda \asymp \frac{(\zeta \vee \sigma)}{\sqrt{m_1m_2}} \frac{a + \zeta \sqrt{m_1m_2}}{1+a} \sqrt{\frac{Ld \log d}{n}}$, where $d = m_1+m_2$, 
 then for any $n \gtrsim d\log d$ and $a = \smallO((m_1m_2)^{1/4})$, there exist two  constants $C_8$ and $C_9$ only depending on $c_0$ such that the estimator $\hat A$ from \eqref{est:hatA} satisfies
    \begin{multline}
        \frac{1}{m_1m_2}\|\hat{A} - A_0\|_F^2 \leq C_8 \nu (\zeta^2 \vee \sigma^2) \rank(A_0) \sqrt{\frac{Ld\log d}{n}}\\
        + C_9 \nu \zeta^2\sqrt{\frac{L\log d}{n}},
    \end{multline} 
    with probability at least $1-(\kappa+1)/d$. Using the condition that $d\log d = \bigO( n)$, we have 
    \begin{equation}\label{eq:order-comp}
        \mathrm{rank}(\hat{A}) = \mathcal{O}_p (\rank(A_0)).   
    \end{equation}
\end{corollary}

Corollary \ref{corollary1} provides a concrete upper bound for error estimation in the Frobenius norm and establishes a bound on the rank of the estimator by selecting appropriate values of $\lambda$ and $a$. When the sample size $n$ increases, the upper bound of the error decreases at a rate of $(d\log d/n)^{1/2}$, and the rank remains in a constant order, independent of the size of the matrix. This constant order bound of the rank is significantly better than the worst-case bound, which is $\bigO(m)$. %\yl{[did any paper give worst case bound?]}. 
This result is aligned with the rank of the true matrix that is bounded by a constant. Additionally, it is worth mentioning that our proof techniques, detailed in  Appendix \ref{A3:Them4,Cor1}, can be utilized for other non-convex regularizations, such as SCAD regularization \citep{fan2001variable},  
MCP regularization \citep{zhang2010nearly}, $L_p$ norm \citep{chartrand2007exact}, and $L_1/L_2$ functional \citep{rahimi2019scale} on the singular values of matrices.  %More details can be found  %Appendix \ref{A3:Them4,Cor1}. 
One limitation of Corollary \ref{corollary1} is that \eqref{eq:order-comp} does not imply the ideal scenario when $\mathrm{rank}(\hat{A}) = \rank(A_0)$. In future work, we will investigate whether this oracle property can be achieved by properly choosing $a$ and $\lambda$.

\section{SIMULATION STUDY}
\label{sec:4.1}

In this simulation study, we generate the target matrix $A_0 \in \mathbb{R}^{m_1 \times m_2}$ as the product of two matrices of smaller dimensions, i.e., $A_0 = UV^\intercal$, where $U\in \mathbb{R} ^{m_1 \times r}$, $V\in \mathbb{R} ^{m_2 \times r}$ with each entry of $U$ and $V$ independently sampled from a standard normal distribution $\mathcal{N}(0, 1)$. As a result, the rank of $A_0$ is at most $r$, which is significantly smaller than $\mathrm{min}(m_1, m_2)$. 

We adopt three sampling schemes outlined in \cite{fang2018max} to facilitate a direct comparison of completion results. All the sampling schemes are implemented without replacement. The first scheme involves uniform sampling of the indices of observed entries, while the subsequent two schemes sample indices in a non-uniform manner. 
Specifically, for each $(k,l) \in [m_1]\times [m_2]$,  let $p_k$ (and $p_l$) be:
\begin{table}[htbp]
    \vspace{-2mm}
    \centering
    %\caption{Different Schemes}
    \small
    \begin{tabular}{p{0.14\columnwidth} p{0.35\columnwidth} p{0.35\columnwidth}}
        \toprule
        \text{Scheme 1} & \text{Scheme 2} & \text{Scheme 3} \\
        \midrule
        $p_k =$ & $p_k =$ & $p_k =$ \\
        $\frac{1}{m_1 m_2}$ & 
        $\begin{cases} 
            2p_0 \; \text{if } k \leq \frac{m_1}{10} \\
            4p_0 \; \text{if } \frac{m_1}{10} < k \leq \frac{m_1}{5} \\
            p_0 \;\;\; \text{otherwise}
        \end{cases}$ & 
        $\begin{cases} 
            3p_0 \; \text{if } k \leq \frac{m_1}{10} \\
            9p_0 \; \text{if } \frac{m_1}{10} < k \leq \frac{m_1}{5} \\
            p_0 \;\;\; \text{otherwise}
        \end{cases}$ \\
        \bottomrule
    \end{tabular}
    \label{tab:schemes}
    \vspace{-5mm}
\end{table}

where \( p_0 \) is a normalized constant such that \( \sum_{k=1}^{m_1} p_k = 1 \). Let $\pi_{kl}$=$p_kp_l,$ then the sampling distribution is $\Pi=(\pi_{kl})_{k,l=1}^{m_1, m_2}$. 
Note that the configuration of Scheme 3 demonstrates a higher level of imbalance, as some entries have higher probabilities, meaning they are more likely to be observed than others. This imbalance poses a greater challenge for recovery efforts compared to Scheme 1 and Scheme 2, as evidenced by more substantial errors in Table \ref{tab:scheme3} than Tables \ref{tab:scheme1} and \ref{tab:scheme2}.

From the trace regression model \eqref{trace_reg_mod}, we define sampling ratios as $\mathrm{SR} = n/(m_1m_2)$ and set the value of $\sigma$ such that the signal-to-noise ratio, defined by $\mathrm{SNR}= 10 \mathrm{log}10 (\frac 1{\sigma^2}\sum_i^n \langle T_i, A_0\rangle^2),$ is either 10 or 20.
To quantitatively evaluate the matrix recovery performance, we use the relative error (RE) defined by
    $\mathrm{RE} = \|\hat{A} - A_0\|_F /  \|A_0\|_F $,
where $\hat{A}$ is an estimator of $A_0$. To find $\hat A$, we compare the TL1-regularized model \eqref{est:hatA} with nuclear norm \citep{candes2012exact}, max-norm \citep{cai2016matrix}, and a hybrid approach combining both the nuclear norm and max-norm \citep{fang2018max}. 
We explore various combinations of dimensions, ranks, sampling schemes, and sampling ratios (SR) under noisy cases. For each combination, we select the optimal parameters for each competing method; please refer to Appendix \ref{A6:tune_para} for more details on parameter tuning. Due to space limitations, we only display the results for SNR = 10 here. Please refer to Appendix \ref{A5:Other_res} for the results of SNR = 20.

We present the recovery comparison under the three sampling schemes in Tables \ref{tab:scheme1}, \ref{tab:scheme2}, and \ref{tab:scheme3}, respectively.
In Table \ref{tab:scheme1} where uniformly sampled data is used, our results show that the TL1 regularization method yields the most favorable recovery outcomes in a noisy setting. Notably, the hybrid approach outperforms the max-norm method, aligning with established numerical analyses in \cite{fang2018max} that posit the inferior performance of the max-norm method compared to the nuclear norm when the observed entries are indeed uniformly sampled. 
In Tables \ref{tab:scheme2} and \ref{tab:scheme3} where non-uniform sampling distributions are employed, we find that TL1 regularization performs best with a few exceptions. 
%These results are consistent with our expectations \yl{[mention any related theroems?]}. 
Furthermore, the results in Table \ref{tab:scheme3} are generally worse than the ones in Table \ref{tab:scheme2}, which indicates the recovery difficulty of Scheme 3 for matrix completion. In short, we conclude that the TL1-regularized approach is robust across various noise levels and sampling distributions.

\begin{table*}
    \vspace{-5mm}
    \centering
    \caption{Relative errors of the reconstructed matrix to the ground truth under Scheme 1 setting with %noiseless and 
    SNR=10. Each reported value is the average RE of over 50 random realizations, with standard deviation in parentheses. We highlight the best values (smallest REs) using boldface and 2nd best in italics.}\label{tab:scheme1} 
    {\fontsize{9.5}{10}\selectfont
    \begin{tabular}{cccccccccc}
    \toprule
    %   \multicolumn{10}{c}{Scheme 1 with SNR=10} \\
    % \midrule
      &(r, SR) & \multicolumn{2}{c}{Max-norm} & \multicolumn{2}{c}{Hybrid} & \multicolumn{2}{c}{Nuclear} & \multicolumn{2}{c}{TL1} \\ 
    \cline{3-10} 
      & &RE & Time & RE & Time & RE & Time  & RE & Time \\
    \midrule
      300 &(5, 0.1)& 0.350 (0.011) &16.14 & 0.217 (0.008)& 17.24& \textit{0.214 (0.008)}& 7.82& \textbf{0.076 (0.002)} &7.15\\
      &(5, 0.2)& 0.200 (0.005)& 19.64& 0.095 (0.002)& 22.10 & \textit{0.085 (0.001)}& 6.32& \textbf{0.045 (0.001)} &7.45\\
      &(10, 0.1)& 0.585 (0.011)& 24.12& \textit{0.531 (0.011)}&24.07& 0.534 (0.010)&6.54& \textbf{0.161 (0.003)} &7.84\\
      &(10, 0.2)& 0.260 (0.006)& 20.97& \textit{0.157 (0.003)}& 23.06& 0.159 (0.003)& 6.56& \textbf{0.071 (0.001)}& 7.28\\
    \midrule
      500&(5, 0.1)& 0.225 (0.007) & 49.79 & 0.118 (0.002)& 50.02 & \textit{0.115 (0.002)}& 25.82 & \textbf{0.076 (0.001)} & 28.77\\
      &(5, 0.2)& 0.134 (0.004) & 48.18 & 0.079 (0.000)& 58.31& \textit{0.073 (0.001)}& 25.88 & \textbf{0.068 (0.001)} & 29.18\\
      &(10, 0.1)& 0.324 (0.015)& 49.21& 0.244 (0.005)& 49.96& \textit{0.242 (0.005)}& 25.77&  \textbf{0.092 (0.001)} &29.18\\
      &(10, 0.2)& 0.164 (0.003) & 48.32 & 0.103 (0.001)& 59.25& \textit{0.098 (0.001)}& 25.95 & \textbf{0.081 (0.001)} & 29.33\\
    \bottomrule
    \end{tabular}
    }
\end{table*}

\begin{table*}
  \centering
  \caption{Relative errors of the reconstructed matrix to the ground truth under Scheme 2 with SNR=10.}\label{tab:scheme2}
 {\fontsize{9.5}{10}\selectfont
    \begin{tabular}{cccccccccc}
    % \toprule
    % \multicolumn{10}{c}{Scheme 2 without noise} \\
    % \midrule
    %   &(r, SR) & \multicolumn{2}{c}{Max-norm} & \multicolumn{2}{c}{Hybrid} & \multicolumn{2}{c}{Nuclear} & \multicolumn{2}{c}{TL1} \\ 
    % \cline{3-10} 
    %   & &RE & Time & RE & Time & RE & Time  & RE & Time \\
    % \midrule
    %   300&(5, 0.1)& 0.304 (0.004)& 22.56& \textbf{0.280 (0.004)}& 22.68& 0.759 (0.002)& 4.92& \textit{0.367 (0.006)} &5.86\\
    %   &(5, 0.2)& 0.173 (0.002)& 24.53& \textit{0.137 (0.003)}& 23.49& 0.606 (0.002)& 5.12& \textbf{0.062 (0.006)} &5.94\\
    %   &(10, 0.1)& 0.479 (0.003)& 25.69& \textbf{0.477 (0.004)}& 24.38& 0.798 (0.001)& 5.20& \textit{0.502 (0.004)} &7.43\\
    %   &(10, 0.2)& 0.216 (0.002)& 25.60& \textit{0.202 (0.003)}& 24.62& 0.610 (0.002)& 5.25& \textbf{0.138 (0.004)} &8.31\\
    % \midrule

    %   500&(5, 0.1)& 0.264 (0.003) & 60.98& \textit{0.209 (0.003)}& 48.44& 0.752 (0.002)& 18.45& \textbf{0.046 (0.005)} &21.87\\
    %   &(5, 0.2)& 0.143 (0.001)& 60.30& \textit{0.125 (0.001)}& 47.55& 0.596 (0.002)& 19.40& \textbf{0.022 (0.000)} &22.56\\
    %   &(10, 0.1)& 0.337 (0.002)& 59.34& \textit{0.285 (0.002)}& 50.50& 0.762 (0.001)& 19.17& \textbf{0.129 (0.003)} &21.99\\
    %   &(10, 0.2)& 0.177 (0.001)& 59.11& \textit{0.144 (0.001)}& 49.42& 0.601 (0.001)& 19.55& \textbf{0.009 (0.002)} & 23.06\\
    % \bottomrule\\

    \toprule
    % \multicolumn{10}{c}{Scheme 2 with SNR=10}\\
    % \midrule
      &(r, SR) & \multicolumn{2}{c}{Max-norm} & \multicolumn{2}{c}{Hybrid} & \multicolumn{2}{c}{Nuclear} & \multicolumn{2}{c}{TL1} \\ 
    \cline{3-10} 
      & &RE & Time & RE & Time & RE & Time  & RE & Time \\
    \midrule
      300&(5, 0.1)& 0.347 (0.004)& 22.57& \textbf{0.338 (0.003)}& 21.02& 0.767 (0.002)& 4.79& \textit{0.402 (0.006)} &8.40\\
      &(5, 0.2)& 0.213 (0.002)& 24.68& \textit{0.193 (0.003)}& 24.56& 0.611 (0.003)& 5.66& \textbf{0.114 (0.003)} &8.50\\
      &(10, 0.1)& 0.521 (0.003)& 25.72& \textbf{0.517 (0.003)}& 21.94& 0.806 (0.001)& 4.96& \textit{0.560 (0.003)}& 7.87\\
      &(10, 0.2)& 0.265 (0.002)& 24.01& \textit{0.264 (0.003)}& 17.06& 0.620 (0.002)& 6.79& \textbf{0.188 (0.003)} &8.15\\
    \midrule
      500&(5, 0.1)& 0.282 (0.004)& 59.46& \textit{0.251 (0.003)}& 51.13& 0.758 (0.002)& 19.23& \textbf{0.106 (0.002)} &22.11\\
      &(5, 0.2)& 0.161 (0.001)& 59.66& \textit{0.149 (0.001)}& 55.75& 0.604 (0.002)& 21.55& \textbf{0.066 (0.000)} &24.36\\
      &(10, 0.1)& 0.365 (0.002)& 57.83& \textit{0.331 (0.002)}& 50.83& 0.771 (0.001)& 19.56& \textbf{0.197 (0.003)} &22.31\\
      &(10, 0.2)& 0.201 (0.000)& 58.72& \textit{0.181 (0.000)}& 57.62& 0.613 (0.001)& 21.51& \textbf{0.087 (0.001)} &24.02\\
    \bottomrule
    \end{tabular}
    }
\end{table*}

\begin{table*}
  \centering
  \caption{Relative errors of the reconstructed matrix to the ground truth under Scheme 3 with SNR=10.}\label{tab:scheme3}
{\fontsize{9.5}{10}\selectfont
 \begin{tabular}{cccccccccc}
    % \hline
    %   \multicolumn{6}{c}{Scheme 3 without noise}\\
    % \hline
    %    & (r, SR) & Max-norm & Hybrid & Nuclear & TL1\\
    % \hline
    %   300&(5, 0.1)& 0.465 (0.021)& \textit{0.450 (0.023)}& 0.784 (0.010)& \textbf{0.445 (0.030)}\\
    %   &(5, 0.2)& 0.184 (0.019)& \textit{0.150 (0.024)}& 0.610 (0.017)& \textbf{0.095 (0.037)}\\
    %   &(10, 0.1)& 0.532 (0.018)& \textit{0.503 (0.018)}& 0.801 (0.007)& \textbf{0.550 (0.018)}\\
    %   &(10, 0.2)& 0.227 (0.019)& \textit{0.219 (0.012)}& 0.624 (0.012)& \textbf{0.146 (0.025)}\\
    % \hline
    %   500&(5, 0.1)& 0.392 (0.025)& \textit{0.323 (0.026)}& 0.756 (0.014)& \textbf{0.106 (0.030)}\\
    %   &(5, 0.2)& 0.157 (0.008)& \textit{0.136 (0.007)}& 0.617 (0.014)& \textbf{0.082 (0.002)}\\
    %   &(10, 0.1)& 0.507 (0.012)& \textit{0.471 (0.013)}& 0.792 (0.007)& \textbf{0.328 (0.015)}\\
    %   &(10, 0.2)& 0.187 (0.008)& \textit{0.159 (0.009)}& 0.633 (0.010)& \textbf{0.099 (0.022)}\\
    %  \hline\\

    \toprule
    % \multicolumn{10}{c}{Scheme 3 with SNR=10}\\
    % \midrule
      &(r, SR) & \multicolumn{2}{c}{Max-norm} & \multicolumn{2}{c}{Hybrid} & \multicolumn{2}{c}{Nuclear} & \multicolumn{2}{c}{TL1} \\ 
    \cline{3-10} 
      & &RE & Time & RE & Time & RE & Time  & RE & Time \\
    \midrule
      300&(5, 0.1)& 0.488 (0.018)& 18.92& \textbf{0.487 (0.020)}& 17.29& 0.790 (0.010)& 5.78& \textit{0.490 (0.027)} &6.35\\
      &(5, 0.2)& 0.219 (0.018)& 19.53& \textit{0.207 (0.019)}& 17.98& 0.617 (0.016)& 6.60& \textbf{0.122 (0.021)} &7.11\\
      &(10, 0.1)& 0.650 (0.015)& 19.19& \textit{0.655 (0.016)}& 16.97& 0.833 (0.007)& 5.80& \textbf{0.572 (0.016)} &6.40\\
      &(10, 0.2)& 0.279 (0.016)& 19.28& \textit{0.267 (0.018)}& 17.72& 0.638 (0.011)& 6.58& \textbf{0.208 (0.024)} &7.12\\
    \midrule
      500&(5, 0.1)& 0.406 (0.023)& 63.43&  \textit{0.350 (0.023)}& 61.82& 0.764 (0.013)& 60.88& \textbf{0.171 (0.026)} &60.26\\
      &(5, 0.2)& 0.177 (0.008)& 53.47& \textit{0.159 (0.007)}& 50.34& 0.611 (0.014)& 50.16 & \textbf{0.098 (0.001)} &50.45\\
      &(10, 0.1)& 0.522 (0.011)& 19.34& \textit{0.493 (0.012)}& 20.54& 0.798 (0.007)& 19.12& \textbf{0.380 (0.015)} &20.81\\
      &(10, 0.2)& 0.296 (0.007)&22.53 & \textit{0.214 (0.009)}& 23.68& 0.657 (0.010)& 22.05& \textbf{0.128 (0.005)}&23.53\\
    \bottomrule
    \end{tabular}
    }
\end{table*}

\begin{table*}
    \centering
    \caption{The rank estimation while keeping the smallest error for different approaches under Scheme 2 with SNR=10.}
    \label{tab:discuss}
    \begin{tabular}{>{\centering\arraybackslash}p{0.17\columnwidth} >{\centering\arraybackslash}p{0.14\columnwidth} >{\centering\arraybackslash}p{0.20\columnwidth} >{\centering\arraybackslash}p{0.11\columnwidth} >{\centering\arraybackslash}p{0.11\columnwidth} >{\centering\arraybackslash}p{0.11\columnwidth}}
    \toprule
        $m_1=m_2$ & (r, SR) & Max-norm & Hybrid & Nuclear & TL1 \\
    \midrule
        500 & (5, 0.2) & 42 & 36 & 58 & 28 \\
    \bottomrule
    \end{tabular}
\end{table*}

\begin{table*}
    \centering
    \caption{Comparison in terms of TRMSE and estimated rank on the Coat Shopping Dataset and MovieLens 100K Dataset. We highlight the best TRMSE values using boldface and 2nd best in italics.} 
    \label{tab:realdata}
    \begin{tabular}{ccccc | cccc}
        \toprule
        \multicolumn{5}{c|}{Coat Shopping Dataset} & \multicolumn{4}{c}{MovieLens 100K Dataset} \\
        \midrule
        Method & Max-norm & Hybrid & Nuclear & $\text{TL1}$ &  Max-norm & Hybrid & Nuclear & $\text{TL1}$ \\
        \midrule
        TRMSE & 1.0635 & \textit{1.0107} & 1.1726 & \textbf{0.9975} & \textit{1.0128} & 1.0375 & 1.1586 & \textbf{1.0051} \\
        Rank & 28 & 30 & 33 & 22 & 38 & 35 & 59 & 30 \\
        \bottomrule
    \end{tabular}
\end{table*}

\paragraph{Discussion 1.}

We explore how the hyper-parameter $a$ in the TL1 regularization affects the performance of the model \eqref{est:hatA} for low-rank matrix completion.  In particular, we examine the rank estimations of different methods for a $500 \times 500$ target matrix of rank 5 under Scheme 2 with 20\% sampling rate and SNR = 10 in one simulation. We set the hyper-parameter $a=100$ for TL1 regularization, which is mostly chosen among all the simulation scenarios. 
It is evident in Table \ref{tab:discuss} that the TL1-regularized model results in the lowest rank for the estimator while keeping the smallest relative error. 

\begin{figure}[h]
\vspace{-4mm}
    \centering 
    \caption{Impact of the parameter $a$ in TL1 on the matrix recovery: relative errors (left) and estimated rank (right) with respect to $a$.} 
    %\yl{[The resolution is not good; maybe you should generate eps file]}}
    \includegraphics[scale=0.4]{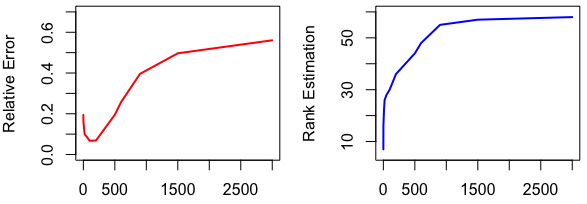}
    \setlength{\abovecaptionskip}{-5pt}  % Adjust this value as needed to reduce space above the caption
    \setlength{\belowcaptionskip}{-2pt}  % Adjust this value as needed to reduce space below the caption
    \label{fig:discuss}
    \vspace{-5mm}
\end{figure}

Figure \ref{fig:discuss} illustrates the changes in the relative error and
the rank of the estimator matrix with respect to the parameter $a$. For each $a,$ we select the optimal value of $\lambda$ in the model \eqref{est:hatA}. 
We observe in Figure \ref{fig:discuss} that the relative error decreases as $a$ increases from 0, reaching a minimum around $a = 100$, and then increases steadily, eventually becoming asymptotically flat. For the rank estimation, there is an increasing trend as the value of $a$ increases. The curve rises sharply at first, indicating that smaller values of $a$ lead to better rank estimation. When $a$ is relatively large, the estimated rank becomes close to that of the nuclear norm method. This empirical behavior aligns with the TL1's property in \eqref{limits}. In Table \ref{tab:discuss},
we present the rank estimated by all competing methods under the same setting, showing that the TL1 regularization produces a matrix with the lowest rank among them. Although the rank estimated by TL1 regularization is higher than the true rank (i.e., 5), this outcome is reasonable due to the trade-off between estimation error and rank.

\paragraph{Discussion 2.}
We compare the bias and variance of estimators obtained from TL1 regularization and nuclear norm regularization. The TL1-regularized model is expected to introduce less bias than nuclear norm regularization, as TL1 reduces the penalization on larger singular values, potentially leading to higher accuracy in matrix completion. In contrast, the nuclear norm-based model minimizes the sum of singular values, which generally reduces the estimator's variance but tends to introduce higher bias. To empirically validate these hypotheses, we design an experiment under Scheme 2 with SNR$=10$ over 100 random trials. The results, presented in Table \ref{tab_d2} in Appendix \ref{A5:Other_res}, show that the estimator  from the TL1-regularized model exhibits a lower bias but slightly higher variance compared to   the nuclear norm regularized model.
%\yl{[add a conclusion about the results]}

\section{REAL DATA APPLICATIONS}

We investigate the performance of two real-world datasets: Coat Shopping and MovieLens 100K. Please refer to Appendix \ref{A7.real-data} 
for data description.
For both datasets, we randomly divide the test set into two distinct subsets: a validation set for the purpose of tuning parameters and an evaluation set to assess the performance of the estimator. Each subset includes 50 percent of the observed entries from the original test set. 
We use the test root mean squared error (TRMSE) restricted on the evaluation set to gauge the recovery performance, as outlined in \cite{wang2021matrix}.
We report the TRMSE values and the ranks of the estimators in Table \ref{tab:realdata}, showing that TL1 outperforms the other approaches in terms of TRMSE for both datasets.
For the Coat Shopping Dataset, the hybrid-regularized method provides the second-smallest TRMSE, but it produces a higher rank than the TL1-based method. In contrast, the nuclear norm regularized model performs the worst. For the MovieLens dataset, TL1 regularization delivers the best TRMSE with the lowest rank. In summary, the TL1-regularized model demonstrates significant improvement over other methods with consistent smallest errors and relatively lower rank estimation.

\section{CONCLUSION}

This paper makes a %substantial %thorough
%\jiayi{maybe change the word ``through"} 
contribution to the statistical analysis for controlling the rank of the estimator and the non-asymptotic upper bounds of recovery errors produced by the TL1 regularization.
For an asymptotically large parameter $a$, the error bound achieves a minimax optimal convergence rate up to a logarithmic factor for low-rank or approximately low-rank matrices, aligned with existing literature work on nuclear norm regularization. Regarding the theoretical analysis of rank estimation which has not yet been extensively studied, we establish a constant order bound for exactly low-rank matrices when $a$ is small, and we aim to refine this bound in future research. 
%achieving the minimax optimal convergence rate up to a logarithmic factor \kun{detailed in Theorem \ref{Them1} and Theorem \ref{Them2}}. \kun{Our analysis also presents the limitations, particularly concerning the error bound discussed in Theorem \ref{Them3} and the rank estimation outlined in Corollary \ref{corollary1}. Neither result is optimal, and we will make further improvements in future studies.}
Overall, our work %admirably 
bridges the gap in the literature, which primarily addresses matrix completion with uniform sampling, by extending it to a more general sampling mechanism. Experimental results demonstrate that TL1 %significantly 
outperforms other methods with consistently smaller errors and lower ranks regardless of whether the sampling scheme is uniform. We include an empirical study of the parameter $a$ for its impact on the matrix completion, aiming at some guidance in the practical use of TL1 regularize.

\section*{Acknowledgments}
The authors thank the reviewers for their helpful comments and suggestions. Portions of this research
were conducted with the computing resources provided by UTD High Performance Computing (HPC). The work of Jiayi Wang is partly supported by the National Science Foundation (DMS-2401272). Yifei Lou is partly supported by the NSF CAREER Award 2414705.

%\newpage
\bibliographystyle{plainnat}
\bibliography{references}

\onecolumn

\aistatstitle{Nosiy Low-Rank Matrix Completion via Transformed $L_1$ Regularization and its Theoretical Properties: \\
Supplementary Materials}

\appendix

% \begin{center}
%     \Large{\textbf{Supplementary Materials}}
% \end{center}

\section{Numerical algorithm}
\label{A4:Algorithm}

Here we describe an efficient algorithm for solving the TL1-regularized model for matrix completion, defined in \eqref{est:hatA}. Recall that we are given 
$n$  independent pairs $(T_i, Y_i)$ with $T_i\in \Gamma$ and $Y_i\in \mathbb R$ in the trace regression model \eqref{trace_reg_mod}, which can be expressed by the Hadamard product, i.e.,
\begin{equation}
    Y = T \circ A_0 + N,
\end{equation}
where $T = \sum_{i=1}^n T_i$, $\circ$ denotes the elementwise Hadamard product, $N$ is the noise term. The resulting matrix $Y$ is of the same size as the underlying matrix $A_0\in\mathbb R^{m_1\times m_2}.$ Consequently, the optimization problem \eqref{est:hatA} can be reformulated as:
\begin{equation}\label{opt_prob}
    \begin{aligned}
        &\underset{A,Z}{\min} \frac{1}{n} \| Y - T \circ A \|_F^2 + \lambda \text{TL1}_a (Z)\\
        &\text{subject to} \; A = Z \; \text{and} \; \|A\|_\infty \leq \zeta, 
    \end{aligned}
\end{equation}
with an auxiliary variable $Z \in \mathbb{R}^{m_1 \times m_2}$ so that we can apply the alternating direction method of multipliers (ADMM) \citep{boydPCPE11admm}.
In particular, the augmented Lagrangian function of \eqref{opt_prob} can be written as
\begin{equation}
    \begin{aligned}
        \mathcal{L}(A, Z, W) = \frac{1}{n}\| Y - T \circ A \|_F^2 + \lambda \text{TL1}_a (Z) + \langle A-Z, W\rangle+\frac{\rho}{2} \|A-Z\|_F^2, 
    \end{aligned}
\end{equation}
where $W \in \mathbb{R}^{m_1 \times m_2}$ is a dual variable and $\rho > 0$ is a step size.
The ADMM scheme involves iteratively minimizing the augmented Lagrangian with respect to $A$ while keeping $Z$ and 
$W$ fixed, then minimizing with respect to $Z$ while keeping $A$ and $W$ fixed, and finally performing a gradient ascent step with respect to $W$ while keeping $A$ and $Z$ fixed. In short, ADMM iterates as follows,
\begin{align}
 \label{ADMM_A}   &A^{k+1} = \underset{\|A\|_\infty \leq \zeta}{\arg\min} \; \mathcal{L}(A, Z^k, W^k),\\
  \label{ADMM_Z}  &Z^{k+1} = \underset{Z}{\arg\min} \; \mathcal{L}(A^{k+1}, Z, W^k),\\
 \label{ADMM_V}   &W^{k+1} = W^{k} + \tau\rho(A^{k+1} - Z^{k+1}),
\end{align}
where $\tau \in \left( 0, (1+\sqrt{5})/2 \right)$ is a parameter to acceleration. We follow the work of \cite{fang2018max} to set $\tau=1.618$.

For the $A$ sub-problem \eqref{ADMM_A}, we write it as
\begin{align}\label{eq:A-sub}
    A^{k+1} = \underset{\|A\|_\infty \leq \zeta}{\arg\min} \; \frac{1}{n}\| Y - T \circ A \|_F^2 + \frac{\rho}{2} \|A-Z^k+\frac{1}{\rho}W^k\|_F^2.
\end{align}
Without the constraint $\|A\|_\infty \leq \zeta,$ the optimal solution to \eqref{eq:A-sub} can be expressed by
\begin{equation}
    A^{k+\frac 1 2} := \left(\frac 2 n T\circ Y+ \rho Z^{k}- W^{k}\right)\oslash (\frac 2 n T + \rho),
\end{equation}
where $\oslash$ denotes the elementwise division. Then we project the solution to the constraint $[-\zeta, \zeta]$, thus leading to
\begin{align}
    A^{k+1} = \min\left\{\max \{A^{k+\frac 1 2},\zeta\}, -\zeta \right\},
\end{align}
where min and max are conducted elementwise.

% and after simplifying this equation, it appears as: 

% where $I \in \mathbb{R}^{m_1 \times m_2}$ is an identity matrix, $J \in \mathbb{R}^{m_1 \times m_2}$ is a matrix of ones, and here min($\cdot$) and max($\cdot$) are element-wise comparison for matrices.

The $Z$ subproblem \eqref{ADMM_Z} can be formulated by
\begin{align}\label{eq:Z-sub}
    Z^{k+1} = \underset{Z}{\arg\min} \; \lambda \text{TL1}_a (Z) + \frac{\rho}{2} \|X^{k+1}-Z+\frac{1}{\rho}W^k\|_F^2,
\end{align}
which has a closed-form solution that is similar to the singular value thresholding (SVT) operator for the nuclear norm minimization \citep{cai2010singular}. 
 Specifically, we define the proximal operator \citep{zhang2015transformed} for the TL1 regularization applied on a scalar $x$ to be
 \begin{align*}
     \text{prox}_{\text{TL1}_a}(x, \mu) := & \arg\min_{z\in\mathbb R}\left\{\mu\frac{(a+1)z}{a+z} + \frac 1 2 (z-x)^2\right\}\\
    =& \mathrm{sign}(x) \left\{ \frac{2}{3}(a+|x|)\cos(\frac{\phi(x)}{3}) - \frac{2a}{3}+\frac{|x|}{3} \right\},
\end{align*}
with $\phi(x) = \arccos(1- \frac{27\mu a(1+a)}{2(a+|x|)^3}).$  Then, the closed-form solution for the $Z$-update \eqref{eq:Z-sub} is given by
\begin{align}
    Z^{k+1} = U\mbox{diag}\left(\{\text{prox}_{\text{TL1}_a}(\sigma_k, \lambda/\rho)\}_{1\leq k\leq m} \right)V^\intercal,
\end{align}
 where we have  the singular value decomposition (SVD) of the matrix $A^{k+1} + \frac{1}{\rho}W^k = U\Sigma V^\intercal$ and the diagonal matrix $\Sigma$ has elements $\sigma_k$ for $1\leq k\leq m.$
 
%  denote  where $U$ and $V$ are the left and right singular vectors, respectively,  and $\Sigma$ is a diagonal matrix. We further 

%  , i.e., $  X =   U  \Sigma   V^*$, we denote 
% \begin{equation}\label{eq:SVT}
% \mbox{SVT}( X,\mu) :=   U  \Sigma_\mu   V^*, \   \Sigma_\mu=\mbox{diag}(\{\max(\sigma_k-\mu,0)\}_{1\leq k\leq r}),
% \end{equation}
% for any positive parameter $\mu$.

% and $Z^{k+1}$ update can be induced as 

% where 
% \begin{align*}
%     \mathrm{Shrinkage}_{\lambda, \rho, a}(s) = \mathrm{sign}(s) \left\{ \frac{2}{3}(a+s)\mathrm{cos}(\frac{\phi(s)}{3}) - \frac{2a}{3}+\frac{s}{3} \right\},
% \end{align*}
% with $\phi(s) = \mathrm{arcos} \left(1- \frac{27\lambda a(1+a)}{2(a+s)^3}\right)$, $s$ is the singular value of $X^{k+1} + \frac{1}{\rho}V^k$, $\mathrm{sign}(s)$ is the signum function,  of $X^{k+1} + \frac{1}{\rho}V^k$ and $D(\cdot)$ is a function converts a vector into a diagonal matrix; more detailed derivation can be found in \cite{zhang2014minimization}.

% Then, the dual variable is updated by
% \begin{align}
%     V^{k+1} = V^{k} + \tau\rho(X^{k+1} - Z^{k+1}).
% \end{align}

We summarize the overall algorithm in Algorithm \ref{alg:TL1_ADMM}.

\begin{algorithm}
    \caption{TL1-regularized matrix completion \eqref{est:hatA} via ADMM}
    \label{alg:TL1_ADMM}
    \begin{algorithmic}
        \STATE Input: $Y$, $T$
        \STATE Set parameters: $\lambda$, $a$, $\zeta$, $\rho$, $\tau = 1.618$
        \STATE Initialize $Z^0=Y$, $V^0=0$ and $k=0.$
        \WHILE{stopping criterion is not satisfied}
            \STATE $A^{k+1} \gets \min\left\{\max \{\left(\frac 2 n T\circ Y+ \rho Z^{k}- W^{k}\right)\oslash (\frac 2 n T + \rho),\zeta\}, -\zeta \right\}$
            \STATE $Z^{k+1} \gets U\mbox{diag}\left(\{\text{prox}_{\text{TL1}_a}(\sigma_k, \lambda/\rho)\}_{k} \right)V^\intercal,$ where $A^{k+1} + \frac{1}{\rho}W^k = U\text{diag}(\{\sigma_k\}_{k}) V^\intercal$
            \STATE $W^{k+1} \gets W^k + \tau \rho (A^{k+1} - Z^{k+1})$
            \STATE $k \gets k + 1$
        \ENDWHILE
    \end{algorithmic}
\end{algorithm}

\section{Proof of Theorem \ref{Them1} and Theorem \ref{Them3}}
\label{A1:Theorem1,3}

Denote a constraint set:
\begin{align}\label{constrain_set}
    &\mathcal{K}(\zeta, \gamma):= \bigg\{ A\in \mathbb{R}^{m_1 \times m_2}: \|A\|_{\infty} \leq \zeta, \frac{\|A\|_*}{\sqrt{m_1m_2}} \leq \gamma \bigg\},
\end{align}
and a constant
\begin{align}
    Z_{\gamma}  = \underset{A \in \mathcal{K}(\zeta, \gamma)}{\sup} \bigg|\frac{1}{n} \sum_{i=1}^{n} \langle T_i, A \rangle^2 - \|A\|^2_{L_2(\Pi)} \bigg|.
\end{align}
Given an i.i.d. Rademacher sequence $\{\epsilon_i\}_{i=1} ^n$, we define,
 \begin{equation}\label{eq:SigmaSets}
     \Sigma_R =  \frac{1}{n} \sum_{i=1}^{n} \epsilon_i T_i \; \text{and} \; \Sigma =  \frac{\sigma}{n} \sum_{i=1}^{n} \xi_i T_i.
 \end{equation}
We introduce Lemmas \ref{lemma1}-\ref{lemma3} to find a non-asymptotic upper bound on the Frobenius norm error. 

\begin{lemma}\label{lemma1}
    Suppose that $T_i$ are i.i.d. indicator matrices with distribution $\Pi$ on $\Gamma$ for $i = 1,\dots, n,$ and Assumptions 1 and 2 hold. For any $n\geq m \frac{(\log d)^3}{L}$, there exists a constant $C^*$ only depending on $c_0$ such that  
    \begin{align*}
        \mathbb{E}(Z_{\gamma}) \leq C^* \gamma \sqrt{\frac{LM\log d}{n}},
    \end{align*}
    with a probability at least $1-\frac{1}{d}$.
\end{lemma}

\begin{proof}[Proof of Lemma \ref{lemma1}]
    Without loss of generality, we take $\zeta = 1$ in the proposed model \eqref{est:hatA}. Consequently, $\|A\|_{\infty} \leq 1$ indicates $|\langle T_i, A \rangle| \leq 1, \; i = 1, \dots ,n$.   
It follows from the symmetrization inequality from the book \citep[Theorem 2.1]{koltchinskii2011oracle} that
\begin{align*}
    \mathbb{E}(Z_{\gamma}) = \mathbb{E}\bigg( \underset{A \in \mathcal{K}(1, \gamma)}{\sup} \bigg|\frac{1}{n} \sum_{i=1}^{n} \langle T_i, A \rangle^2 - \|A\|^2_{L_2(\pi)} \bigg|\bigg)\leq 2\mathbb{E} \bigg( \underset{A \in \mathcal{K}(1, \gamma)}{\sup} \bigg|\frac{1}{n} \sum_{i=1}^{n} \epsilon_i \langle T_i, A \rangle^2 \bigg| \bigg),
\end{align*}
where $\{\epsilon_i\}_{i=1}^{n}$ are independent Rademacher random variables. We further use the contraction inequality in \citet[Theorem 2.3]{koltchinskii2011oracle} to obtain that
\begin{align*}
    \mathbb{E}(Z_{\gamma}) &\leq 8\mathbb{E} \bigg( \underset{A \in \mathcal{K}(1, \gamma)}{\sup} \bigg|\frac{1}{n} \sum_{i=1}^{n} \epsilon_i \langle T_i, A \rangle \bigg| \bigg) = 8\mathbb{E} \bigg( \underset{A \in \mathcal{K}(1, \gamma)}{\sup}|\langle \Sigma_R, A \rangle| \bigg)\\
    &\leq 8\mathbb{E} \bigg( \underset{A \in \mathcal{K}(1, \gamma)}{\sup} \|\Sigma_R\|\; \|A\|_* \bigg), \quad \text{by the duality between nuclear and operator norms}\\
    &\leq 8\mathbb{E} (\|\Sigma_R\|) \; \gamma\sqrt{m_1m_2}\\
    &\leq 8C^* \sqrt{\frac{L\log d}{mn}} \gamma\sqrt{m_1m_2}, \quad \text{by \citet[Lemma 6]{klopp}}\\
    &\leq C^*\gamma \sqrt{\frac{m_1m_2}{m}} \sqrt{\frac{L\log d}{n}} = C^*\gamma \sqrt{\frac{LM\log d}{n}},
\end{align*}
holds with a probability at least $1-\frac{1}{d}$.
\end{proof}

\begin{lemma}\label{lemma2}
    Suppose that $T_i$ are i.i.d. indicator matrices with distribution $\Pi$ on $\Gamma$ for $i = 1,\dots, n,$ and Assumptions 1 and 2 hold. Then for any $A \in \mathcal{K}(\zeta, \gamma)$, the inequality
    \begin{equation*}
        \frac{1}{n} \sum_{i=1}^{n} \langle T_i, A \rangle^2 \geq \|A\|^2_{L_2(\Pi)} - \zeta^2 \sqrt{\frac{L\log d}{n}} - \zeta \frac{\|A\|_*}{\sqrt{m_1m_2}} \sqrt{\frac{LM\log d}{n}}
    \end{equation*}
holds with probability at least $1- \frac{\kappa}{d}$, where $\kappa$ is a constant depending on a universal constants $K$ and $L$ defined in Assumption 1.
\end{lemma}

\begin{proof}[Proof of Lemma \ref{lemma2}]
    Let $X_i = \langle T_i, A \rangle^2$ and define
\begin{align}\label{eq:V}
    V := n\mathbb{E}(X_i^2) + 16n\gamma \sqrt{M} \sqrt{\frac{L\log d}{n}} \leq n + 16\gamma \sqrt{LMn\log d}\; \lesssim n\sqrt{L\log d}.
\end{align}
By the Talagrand concentration inequality in \citet[Theorem 2.6]{koltchinskii2011oracle} and a logarithmic property that $\log (1+x) \geq \frac{x}{1+x}, \forall x\geq 0$, we take $t \gtrsim \sqrt{\frac{L \log d}{n}}$, thus getting
\begin{equation}\label{talagrand_inq}
\begin{aligned}
    \mathbb{P}\bigg\{ \underset{A \in \mathcal{K}(1, \gamma)}{\sup}|Z_\gamma - \mathbb{E}(Z_\gamma)| \geq t \bigg\} &\leq K \exp \bigg\{ -\frac{1}{K} tn \; \log \bigg(1+\frac{tn}{n\sqrt{L\log d}}\bigg) \bigg\}\\
    &\leq K \exp \bigg\{ -\frac{1}{K} \frac{t^2 n}{\sqrt{L\log d}+t} \bigg\},
\end{aligned}
\end{equation}
where $K$ is a universal constant, $V$ is defined in \eqref{eq:V}, and $U$ is the uniform bound of $X_i$. In our case, $U = 1$. 
%\yl{[revised the order of description; please check]}

We split the proof into two cases depending on the value of $\gamma$.

Case 1: If $\gamma \leq \frac{1}{\sqrt{M}}$, then $\eqref{talagrand_inq}$ implies for $t'\gtrsim 1$
\begin{equation}\label{case1_inq}
\begin{aligned}
     \mathbb{P}\bigg\{\underset{A \in \mathcal{K}(1, \gamma)}{\sup} Z_{\gamma} \geq t'\sqrt{\frac{L\log d}{n}} \bigg\} &\leq K \exp \bigg\{ -\frac{1}{K} \frac{t'^2 L\log d}{\sqrt{L\log d} + t'\sqrt{\frac{L\log d}{n}}} \bigg\}\\
     &= K \exp  \bigg\{ -\frac{1}{K} \frac{t'^2 \sqrt{L\log d}}{1 + \frac{t'}{\sqrt{n}}} \bigg\}
     %&= K \exp \bigg\{ -\frac{1}{K} \sqrt{nd\log d} \frac{t'^2 d\log d}{n+t'\sqrt{nd\log d}} \bigg\}\\
     %&= K \exp \bigg\{ -\frac{1}{K} \frac{t'^2(d\log d)^{\frac{3}{2}}}{\sqrt{n}+t'\sqrt{d\log d}} \bigg\}\\
     \leq K \exp \bigg\{ -\frac{1}{K} \frac{\sqrt{L\log d}}{1+\frac{t'}{\sqrt{n}}} \bigg\}.
\end{aligned}
\end{equation}

Case 2: If $\gamma \geq \frac{1}{\sqrt{M}}$, then, similarly, we have for $t'\gtrsim 1$,
\begin{equation}\label{case2_inq}
    \begin{aligned}
        \mathbb{P}\bigg\{\underset{A \in \mathcal{K}(1, \gamma)}{\sup} Z_{\gamma} \geq t' \gamma \sqrt{\frac{LM\log d}{n}} \bigg\} &\leq K \exp \bigg\{ -\frac{1}{K} \frac{t'^2 \gamma ^2 LM\log d}{\sqrt{L\log d}+t'\gamma\sqrt{\frac{LM\log d}{n}}} \bigg\}\\
        &= K \exp \bigg\{ -\frac{1}{K} \frac{t'^2 \gamma ^2 M \sqrt{L\log d}}{1+\gamma \sqrt{M} \frac{t'}{\sqrt{n}}} \bigg\}\\
        &\leq  K \exp \bigg\{ -\frac{1}{K} \frac{ \sqrt{L\log d}}{1+\gamma \sqrt{M} \frac{t'}{\sqrt{n}}} \bigg\}.
    \end{aligned}
\end{equation}

Next, we use the standard peeling argument to estimate the probability. Specifically, we define a sequence of sets %\yl{[check below for $S_l$ (adding $\|A\|$ inside Sl)]}
\begin{align*}
    S_l = \bigg\{ A \in \mathcal{K}(1, \gamma)  : 2^{l-1} \frac{1}{\sqrt{M}} \leq \frac{\|A\|_*}{\sqrt{m_1m_2}} \leq 2^{l}\frac{1}{\sqrt{M}} \bigg\}, \quad l = 1, 2, \cdots
\end{align*}
Then a series of estimations lead to
\begin{equation}
    \begin{aligned}
        &\mathbb{P} \bigg\{\underset{\|A\|_\infty \leq 1}{\sup} \frac{Z_{\gamma}}{\|A\|_*/\sqrt{m_1m_2}} \geq t' \sqrt{\frac{LM\log d}{n}} \bigg\}\\
        &\leq \sum_{l=1}^{\infty} \mathbb{P} \bigg\{\underset{A \in S_l}{\sup} \frac{Z_{\gamma}}{\|A\|_*/\sqrt{m_1m_2}} \geq t' \sqrt{\frac{LM\log d}{n}} \bigg\}\\
        &\leq \sum_{l=1}^{\infty} \mathbb{P} \bigg\{\underset{A \in \mathcal{K} \big(1, 2^{l}\frac{1}{\sqrt{M}} \big)}{\sup} \frac{Z_{\gamma}}{\|A\|_*/\sqrt{m_1m_2}} \geq t' \sqrt{\frac{LM\log d}{n}} \bigg\}\\
        &\leq \sum_{l=1}^{\infty} \mathbb{P} \bigg\{\underset{A \in \mathcal{K} \big(1, 2^{l}\frac{1}{\sqrt{M}} \big)}{\sup} Z_{\gamma} \geq t' 2^{l-1} \sqrt{\frac{L\log d}{n}} \bigg\}\\
        &\leq \sum_{l=1}^{\infty} K \exp  \bigg\{ -\frac{1}{K} \frac{t'^2 \; 2^{2l-2}\sqrt{L\log d}}{1+\frac{t'}{\sqrt{n}}2^{l-1}} \bigg\}\\
        &\leq \sum_{l=1}^{\infty} K \exp  \bigg\{  -\frac{1}{K} \frac{2^{l-1} \sqrt{L\log d}}{1+ \frac{t'}{\sqrt{n}}} \bigg\}, \quad \text{ because for any $l \geq 1$, $2^{l-1} \geq 1$}\\
        &\leq \sum_{l=1}^{\infty} K \exp  \bigg\{  -\frac{\log 2}{2K} \frac{\sqrt{L\log d}}{1+ \frac{t'}{\sqrt{n}}} \; l \bigg\}, \quad \text{because for any $x>0$, $x > \log x$ always holds}\\
        &= K\frac{\exp  \bigg\{  -\frac{\log 2}{2K} \frac{\sqrt{L\log d}}{1+ \frac{t'}{\sqrt{n}}}\bigg\}}{1-\exp  \bigg\{  -\frac{\log 2}{2K} \frac{\sqrt{L\log d}}{1+ \frac{t'}{\sqrt{n}}}\bigg\}}.
    \end{aligned}
\end{equation}
Based on Case 1 and Case 2, we set $t' = 1$ for large $n$, then for any matrix $A \in \mathbb{R}^{m_1\times m_2}$ with $\|A\|_\infty \leq 1$, the following inequality
\begin{equation*}
    \begin{aligned}
        Z_{\gamma} \leq \sqrt{\frac{L\log d}{n}} + \frac{\|A\|_*}{\sqrt{m_1m_2}} \sqrt{\frac{LM\log d}{n}},
    \end{aligned}
\end{equation*}
holds with at least probability $1- \frac{\kappa}{d}$ where $\kappa$ is a constant depending on $L$.
Therefore, we have
\begin{equation*}
    \begin{aligned}
        Z_{\gamma} \leq \zeta^2 \sqrt{\frac{L\log d}{n}} + \zeta \frac{\|A\|_*}{\sqrt{m_1m_2}} \sqrt{\frac{LM\log d}{n}},
    \end{aligned}
\end{equation*} 
which implies that 
\begin{equation*}
        \bigg|\frac{1}{n} \sum_{i=1}^{n} \langle T_i, A \rangle^2 - \|A\|^2_{L_2(\pi)} \bigg| \leq \zeta^2 \sqrt{\frac{L\log d}{n}} + \zeta \frac{\|A\|_*}{\sqrt{m_1m_2}} \sqrt{\frac{LM\log d}{n}}.
\end{equation*}
A simple calculation leads to the desired inequality
\[
\frac{1}{n} \sum_{i=1}^{n} \langle T_i, A \rangle^2 \geq \|A\|^2_{L_2(\pi)} - \zeta^2 \sqrt{\frac{L\log d}{n}} - \zeta \frac{\|A\|_*}{\sqrt{m_1m_2}} \sqrt{\frac{LM\log d}{n}},
\]
which holds with probability at least $1 - \frac{\kappa}{d}$.
\end{proof}

Lemma \ref{lemma3} is important, as it gives us an upper bound on the Frobenius norm error for the estimator $\hat{A}$ under any general sampling distribution. 

\begin{lemma}\label{lemma3}
    Suppose $T_i$ are i.i.d. indicator matrices with distribution $\Pi$ on $\Gamma$ that satisfies Assumptions 1 and 2 for $i = 1,\dots, n$. Assume $\|A_0\|_\infty \leq \zeta$ for a constant $\zeta$ and Assumption 3 holds. Take $\lambda^{-1} = \mathcal{O} \left( \left( \frac{(\zeta \vee \sigma)}{\sqrt{m_1m_2}} \frac{a + \zeta\sqrt{m_1m_2}}{1+a} \sqrt{\frac{Ld\log d}{n}} \right)^{-1} \right)$,  then for any $n \gtrsim d\log d$, there exist two constants $C_1'$ and $C_2'$ only depending on $c_0$ such that the estimator $\hat A$ from \eqref{est:hatA} satisfies
    \begin{equation}
        \begin{aligned}
            \frac{1}{m_1m_2}\|\hat{A} - A_0\|_F^2 \leq C_1' \nu \left\{ (\zeta \vee \sigma) \sqrt{\frac{Ld\log d}{n}}  \frac{\|A_0\|_{*}}{\sqrt{m_1m_2}} + \lambda \mathrm{TL}1_a(A_0) \right\} + C_2' \nu \zeta^2 \sqrt{\frac{L\log d}{n}},
        \end{aligned}
    \end{equation}
    with probability at least $1- \frac{\kappa+1}{d}$, where $\kappa$ is a constant depending on $L$. %\yl{[Theorems 1 and 3 say "depending on L" make consistent]}
\end{lemma}

\begin{proof}[Proof of Lemma \ref{lemma3}]
    It follows from the optimality of the estimator $\hat{A}$ in \eqref{est:hatA} that
    \begin{align*}
        \frac{1}{n} \sum_{i=1}^{n} (Y_i - \langle T_i, \hat{A} \rangle)^2 + \lambda \text{TL1}_a (\hat{A}) \leq \frac{1}{n} \sum_{i=1}^{n} (Y_i - \langle T_i, A_0 \rangle)^2 + \lambda \text{TL1}_a (A_0).
    \end{align*}
    Replacing $Y_i$ with the trace regression model, we obtain
    \begin{align*}
        \frac{1}{n} \sum_{i=1}^{n} (\langle T_i, A_0 \rangle + \sigma \xi_i -\langle T_i, \hat{A} \rangle)^2 \leq \frac{\sigma^2}{n} \sum_{i=1}^{n} \xi_i^2 + \lambda \text{TL1}_a (A_0) - \lambda \text{TL1}_a (\hat{A}),
        \end{align*}
    which, after expanding the square, is equivalent to, ,    
        \begin{align*}
        % \frac{1}{n} \sum_{i=1}^{n} \langle T_i, A_0 - \hat{A}  \rangle^2 + \frac{\sigma^2}{n} \sum_{i=1}^{n} \xi_i^2 + \frac{2\sigma}{n}\sum_{i=1}^{n} \xi_i \langle T_i, A_0-\hat{A} \rangle \leq \frac{\sigma^2}{n} \sum_{i=1}^{n} \xi_i^2 + \lambda \text{TL1}_a (A_0) - \lambda \text{TL1}_a (\hat{A})\\
        \frac{1}{n} \sum_{i=1}^{n} \langle T_i, \hat{A} - A_0  \rangle^2 \leq \frac{2\sigma}{n}\sum_{i=1}^{n} \xi_i \langle T_i, \hat{A} - A_0\rangle +\lambda \text{TL1}_a (A_0) - \lambda \text{TL1}_a (\hat{A}).
    \end{align*}
    
    By the definition of $\Sigma$ in \eqref{eq:SigmaSets}, we have
    \begin{align} \label{inq:base}
        \frac{1}{n} \sum_{i=1}^{n} \langle T_i, \hat{A} - A_0  \rangle^2 \leq 2\langle \Sigma, \hat{A} - A_0 \rangle +\lambda \text{TL1}_a (A_0) - \lambda \text{TL1}_a (\hat{A}).
    \end{align}
    We further use the duality between the nuclear norm and the operator norm to get
    \begin{align*}
        \frac{1}{n} \sum_{i=1}^{n} \langle T_i, \hat{A} - A_0  \rangle^2 \leq 2\|\Sigma\|\|\hat{A} - A_0\|_* +\lambda \text{TL1}_a (A_0) - \lambda \text{TL1}_a (\hat{A}),
    \end{align*}
    which indicates that
    \begin{align} \label{inq:base_derived}
        \lambda \text{TL1}_a (\hat{A}) \leq 2\|\Sigma\|\|\hat{A} - A_0\|_* +\lambda \text{TL1}_a (A_0).
    \end{align}
It follows from Lemma \ref{lemma2} that
    \begin{equation}
    \begin{aligned}
        &\frac{1}{\nu m_1m_2}\|\hat{A} - A_0\|_F^2 \leq \|\hat{A} - A_0\|^2_{L_2(\Pi)} \\
        &\lesssim \frac{1}{n}\sum_{i=1}^{n} \langle T_i, \hat{A} -A_0 \rangle^2 + \zeta^2 \sqrt{\frac{L\log d}{n}} + \zeta \frac{\|\hat{A} - A_0\|_*}{\sqrt{m_1m_2}} \sqrt{\frac{LM\log d}{n}}\\
        &\lesssim 2\|\Sigma\|\|\hat{A} - A_0\|_* +\lambda \text{TL1}_a (A_0) - \lambda \text{TL1}_a (\hat{A}) + \zeta^2 \sqrt{\frac{L\log d}{n}} + \zeta \frac{\|\hat{A} - A_0\|_*}{\sqrt{m_1m_2}} \sqrt{\frac{LM\log d}{n}}\\
        &\lesssim \left\{ 2\|\Sigma\|  + \frac{\zeta }{\sqrt{m_1m_2}} \sqrt{\frac{LM\log d}{n} } \right\} \|A_0\|_*  + \zeta^2 \sqrt{\frac{L\log d}{n}} + \lambda \text{TL1}_a(A_0)  \\
        &\qquad + \left\{ 2\|\Sigma\|  + \frac{\zeta }{\sqrt{m_1m_2}} \sqrt{\frac{LM\log d}{n} }  - \lambda \frac{1+a}{a + \sigma_1(\hat{A})} \right\} \|\hat{A}\|_* \\
        &\lesssim \left\{ 2\|\Sigma\|  + \frac{\zeta }{\sqrt{m_1m_2}} \sqrt{\frac{LM\log d}{n} } \right\} \|A_0\|_*  + \zeta^2 \sqrt{\frac{L\log d}{n}} + \lambda \text{TL1}_a(A_0)  \\
        &\qquad + \left\{ 2\|\Sigma\|  + \frac{\zeta }{\sqrt{m_1m_2}} \sqrt{\frac{LM\log d}{n} }  - \lambda \frac{1+a}{a + \zeta\sqrt{m_1m_2}} \right\} \|\hat{A}\|_* .
    \end{aligned}
    \end{equation}
The penultimate inequality holds by the triangle inequality of nuclear norm: $\|\hat{A} - A_0\|_* \leq \|\hat{A}\|_* + \|A_0\|_*$ and the inequality property of $\text{TL1}_a$: $\text{TL1}_a(A) \geq \sum_{j=1}^{m} \frac{(a+1)\sigma_j(A)}{a+\sigma_1(A)} = \frac{1+a}{a+\sigma_1(A)}\|A\|_*$. In addition, since $\|\hat{A}\|_\infty \leq \zeta$ and hence $\sigma_1(\hat{A}) \leq \zeta \sqrt{m_1m_2}$,  we have the last inequality. 
Furthermore, using the matrix Bernstein's inequality \citep[Lemma 5]{klopp}, we have for
 $n \gtrsim d$ there exists a constant $C^* > 0$ only depending on $c_0$ such that 
\begin{equation} \label{Bernstein}
    \|\Sigma\| \leq C^* \sigma \sqrt{\frac{L \log d}{mn}} \leq C^* \frac{\sigma}{\sqrt{m_1m_2}} \sqrt{\frac{Ld \log d}{n}},
\end{equation}
holds with the probability at least $1-\frac{1}{d}$.

If $\lambda^{-1} = \mathcal{O} \left( \left( \frac{(\zeta \vee \sigma)}{\sqrt{m_1m_2}} \frac{a + \zeta\sqrt{m_1m_2}}{1+a} \sqrt{\frac{Ld\log d}{n}} \right)^{-1} \right)$, then there exist two constants $C_1'$ and $C_2'$ depending on $c_0$ such that 
\begin{equation}
 %   \begin{aligned}
        % &\frac{1}{\nu m_1m_2}\|\hat{A} - A_0\|_F^2 \leq C_1' \left\{ (\zeta \vee \sigma) \sqrt{\frac{Ld\log d}{n}}  \frac{\|A_0\|_{*}}{\sqrt{m_1m_2}} + \lambda \mathrm{TL}1_a(A_0) \right\} + C_2' \zeta^2 \sqrt{\frac{L\log d}{n}},\\
        \frac{1}{m_1m_2}\|\hat{A} - A_0\|_F^2 \leq C_1' \nu \left\{ (\zeta \vee \sigma) \sqrt{\frac{Ld\log d}{n}}  \frac{\|A_0\|_{*}}{\sqrt{m_1m_2}} + \lambda \mathrm{TL}1_a(A_0) \right\} + C_2' \nu \zeta^2 \sqrt{\frac{L\log d}{n}},
  %  \end{aligned}
\end{equation}
holds with probability at least $1- \frac{\kappa + 1}{d}$, where $\kappa$ is a constant depending on $L$.
\end{proof}

Equipped with Lemma \ref{lemma3}, we are posed to prove Theorem \ref{Them1} and Theorem \ref{Them3}.

\begin{proof}[Proof of Theorem \ref{Them1}]
    Take $\lambda \asymp \frac{(\zeta \vee \sigma)}{\sqrt{m_1m_2}} \frac{a + \zeta \sqrt{m_1m_2}}{1+a} \sqrt{\frac{Ld\log d}{n}}$ and  any $a^{-1} = \mathcal{O}((\zeta \sqrt{m_1m_2})^{-1})$, by the inequality that $\text{TL1}_a(A_0) \leq \frac{1+a}{a} \|A_0\|_*$, we get
    \begin{equation}
        \begin{aligned}
        \lambda \mathrm{TL}1_a(A_0) &\asymp \frac{(\zeta \vee \sigma)}{\sqrt{m_1m_2}} \frac{a + \zeta \sqrt{m_1m_2}}{1+a} \sqrt{\frac{Ld\log d}{n}} \mathrm{TL}1_a(A_0)\\
        &\leq \frac{(\zeta \vee \sigma)}{\sqrt{m_1m_2}} \frac{a + \zeta\sqrt{m_1m_2}}{1+a} \sqrt{\frac{Ld\log d}{n}} \frac{1+a}{a} \|A_0\|_* \\
        &\lesssim (\zeta \vee \sigma) \sqrt{\frac{Ld\log d}{n}} \frac{\|A_0\|_*}{\sqrt{m_1m_2}} \leq (\zeta \vee \sigma) \gamma \sqrt{\frac{Ld\log d}{n}}.
    \end{aligned}
    \end{equation}
   It follows from Lemma \ref{lemma3} that there exist two constants $C_1$ and $C_2$ only depending on $c_0$ such that the estimator $\hat A$ from \eqref{est:hatA} satisfies
    \begin{equation}
        \frac{1}{m_1m_2}\|\hat{A} - A_0\|_F^2 \leq C_1 \nu (\zeta \vee \sigma) \gamma \sqrt{\frac{Ld\log d}{n}} + C_2 \nu \zeta^2 \sqrt{\frac{L\log d}{n}}
    \end{equation}
   holds with probability at least $1- \frac{\kappa + 1}{d}$.
\end{proof}

\tolerance=2000
\emergencystretch=10pt
\begin{proof}[Proof of Theorem \ref{Them3}]
    Since $\text{TL1}_a$ for any $a>0$ is an increasing function with respect to input arguments, it is straightforward that 
    \begin{equation*}
        \text{TL1}_a(A_0) = \sum_{j=1}^{m} \frac{(1+a)\sigma_j(A_0)}{a+\sigma_j(A_0)} \leq \rank (A_0) \frac{(1+a)\sigma_1(A_0)}{a+\sigma_1(A_0)},
    \end{equation*}
which implies that
    \begin{equation*}
        \text{TL1}_a(A_0) \leq \rank(A_0) \frac{(1+a)\zeta \sqrt{m_1m_2}}{a+\zeta \sqrt{m_1m_2}},
    \end{equation*} 
    with $\sigma_1(A_0) \leq \zeta \sqrt{m_1m_2}$. 
 Applying Lemma \ref{lemma3} with $\lambda^{-1} = \mathcal{O} \left( \left( \frac{(\zeta \vee \sigma)}{\sqrt{m_1m_2}} \frac{a + \zeta\sqrt{m_1m_2}}{1+a} \sqrt{\frac{Ld\log d}{n}} \right)^{-1} \right)$, we obtain
    
\begin{comment}
    \begin{equation}
        \begin{aligned}
            \frac{1}{m_1m_2}\|\hat{A} - A_0\|_F^2 &\lesssim \nu (\zeta \vee \sigma) \tau \sqrt{\frac{Ld\log d}{n}} \frac{(\zeta \vee \sigma)}{\sqrt{m_1m_2}} \frac{a+\zeta \sqrt{m_1m_2}}{1+a}  \frac{(1+a)\zeta \sqrt{m_1m_2}}{a+\zeta \sqrt{m_1m_2}} + \nu \zeta^2\sqrt{\frac{L\log d}{n}}\\
            &\lesssim \nu (\zeta^2 \vee \sigma^2) \tau \sqrt{\frac{Ld\log d}{n}} + \nu \zeta^2\sqrt{\frac{L\log d}{n}}.
        \end{aligned}
    \end{equation}
\end{comment}

    \begin{equation}
        \begin{aligned}
            \frac{1}{m_1m_2}\|\hat{A} - A_0\|_F^2 &\lesssim \lambda \nu \frac{(1+a)\zeta \sqrt{m_1m_2}}{a+\zeta \sqrt{m_1m_2}}\rank(A_0) + \nu \zeta^2 \sqrt{\frac{L\log d}{n}} .
        \end{aligned}
    \end{equation}
    
  Consequently, for any $a>0$, there exist two constants $C_5$ and $C_6$ only depending on $c_0$ such that the estimator $\hat A$ from \eqref{est:hatA} satisfies
    \begin{align}
        \frac{1}{m_1m_2}\|\hat{A} - A_0\|_F^2 \leq C_5 \lambda \nu  \frac{(1+a)\zeta \sqrt{m_1m_2}}{a+\zeta \sqrt{m_1m_2}}\rank(A_0) + C_6 \nu \zeta^2\sqrt{\frac{L\log d}{n}},
    \end{align}
    with probability at least $1- \frac{\kappa + 1}{d}$.
\end{proof}

\section{Proof of Theorem \ref{Them2}}
\label{A2:Them2}

To prove Theorem \ref{Them2}, we introduce Lemma \ref{lemma4} along with some definitions. For any matrix $A \in  \mathbb{R}^{m_1 \times m_2}$, let $U_A$ and $V_A$ be the left and right singular matrices of $A$, and $D_A$ is the diagonal matrix with the singular values of $A$, i.e., the SVD of $A$ is expressed by  $ A= U_A D_A V_A^\intercal $. We denote $r_A:=\rank (A)$ and $\sigma_j(A)$ is the $j$th singular values of $A$,  $j =1,\dots, r_A$. 
We define $S_U(A)$ and $S_V(A)$ to be the linear subspaces spanned by column vectors of $U_A$ and $V_A,$ respectively, and denote their corresponding orthogonal components, denoted by $S_U^\perp$ and $S_V^\perp$. We set
\begin{align}
P_A^\perp(B) = \mathbf P_{S_U^\perp(A)}B\mathbf P_{S_V^\perp(A)} \quad \mbox{and} \quad  P_A(B) = B-P_A^\perp(B),
\end{align}
where $\mathbf P_S$ denotes the projection onto the linear subspace $S.$
Then, for any matrix $B \in \mathbb{R}^{m_1 \times m_2}$, there exits the left and right singular vectors $U_B$, $V_B$ and the diagonal matrix $D_B$ with singular values of $P_A^\perp(B)$ such that  $P_A^\perp(B) = U_BD_BV_B^\intercal$; similarly, let $r_B = \rank (P_A^\perp(B))$, and $\sigma_j(P_A^\perp(B))$ is the $j$th singular values of $P_A^\perp(B)$, $j =1,\dots, r_B$.\\

\begin{lemma}\label{lemma4}
   For any two matrices $A$ and $B$, we have
    \begin{equation}
         \mathrm{TL1}_a(A+P_A^\perp(B)) = \mathrm{TL1}_a(A) + \mathrm{TL1}_a(P_A^\perp(B)).
    \end{equation}
\end{lemma}

\begin{proof}[Proof of Lemma \ref{lemma4}]
Some calculations show that
    \begin{equation}
        %\begin{aligned}
            A+P_A^\perp(B) = U_A D_A V_A^\intercal + U_B D_B V_B^\intercal
            = \begin{pmatrix}
                U_A &  U_B
            \end{pmatrix}
            \begin{pmatrix}
                D_A & 0\\
                0 & D_B
            \end{pmatrix}
            \begin{pmatrix}
                V_A^\intercal\\
                V_B^\intercal
            \end{pmatrix}
             = UDV^\intercal,
      %  \end{aligned}
    \end{equation}
    where $U = \begin{pmatrix}
            U_A &  U_B
        \end{pmatrix}$, $D = \begin{pmatrix}
            D_A & 0\\
            0 & D_B
        \end{pmatrix}$, and $V^\intercal = \begin{pmatrix}
            V_A^\intercal\\
            V_B^\intercal
        \end{pmatrix}$.

    Next, we show that $U$ is the left singular vector of $A+P_A^\perp(B)$. It is sufficient to verify that $U^{\intercal}U = I$, which is equivalent to $U_A^{\intercal}U_B = 0$.
    By the definition of $P_A^\perp$, it is straightforward to see that the singular vectors of $P_{A}^\perp(B)$ are orthogonal to the space spanned by the singular vectors of $A$, as SVD guarantees that $U_A$ spans the column space of $A$ and $U_B$ spans the orthogonal complement of the column space of $A$. Therefore, $U_A^\intercal U_B = 0$ as desired. Similarly, we have $V^\intercal V = I $, implying that $V$ is the right singular vector of $A+P_A^\perp(B)$. 
    Hence, $D$ is the corresponding diagonal matrix with singular values of $A+P_A^\perp(B)$.

  Lastly, we apply the $\text{TL}1_a$ function to $A+P_A^\perp(B)$, thus getting
    \begin{equation*}
      %  \begin{aligned}
    \text{TL}1_a(A+P_A^\perp(B)) = \sum_{k=1}^{r_A+r_B}\frac{(1+a)D(k,k)}{a+D(k,k)}= \sum_{i=1}^{r_A}\frac{(1+a)D_A(i,i)}{a+D_A(i,i)} + \sum_{j=1}^{r_B}\frac{(1+a)D_B(j,j)}{a+D_B(j,j)}= \text{TL}1_a(A) + \text{TL}1_a(P_A^\perp(B)).
   %     \end{aligned}
    \end{equation*}
\end{proof}

%\textbf{Proof of Corollary 3}

\begin{proof}[Proof of Theorem \ref{Them2}]
    We start with the derivative of $\text{TL1}_a(A)$ with respect to the matrix $A$.     
    By SVD, we have $A = \sum_{i=1}^{m} \sigma_j(A) u_jv_j^\intercal$, which indicates that $\sigma_j(A) = u_j^\intercal A v_j$, where $u_j$ and $v_j$ are the left and right orthonormal singular vectors of $A$, respectively. Then,
    \begin{equation}\label{eq:derivative}
        \begin{aligned}            \frac{\partial(\text{TL1}_a(A))}{\partial A} &= \frac{\partial}{\partial A} \sum_{j=1}^{m} \frac{(1+a)\sigma_j(A)}{a+\sigma_j(A)} = \sum_{j=1}^{m} \frac{\partial}{\partial(A)} \frac{(1+a)\sigma_j(A)}{a+\sigma_j(A)}
            = \sum_{j=1}^{m} \frac{a(1+a)}{(a+\sigma_j(A))^2} \frac{\partial \sigma_j(A)}{\partial A} \\&= \sum_{j=1}^{m} \frac{a(1+a)}{(a+\sigma_j(A))^2} \frac{\partial (u_j^\intercal A v_j)}{\partial A}
            = \sum_{j=1}^{m} \frac{a(1+a)}{(a+\sigma_j(A))^2} \frac{\partial (\text{tr}(A v_j u_j^\intercal))}{\partial A} = \sum_{j=1}^{m} \frac{a(1+a)}{(a+\sigma_j(A))^2} \frac{u_j v_j^\intercal \partial A}{\partial A} \\
            &=  \sum_{j=1}^{m} \frac{a(1+a)}{(a+\sigma_j(A))^2} u_j v_j^\intercal.
        \end{aligned}
    \end{equation}
    By the Mean Value Theorem, there exists a matrix $\Tilde{A}$ between $\hat{A}$ and $A_0 + P_{A_0}^\perp(\hat{A} - A_0)$ such that
    \begin{equation}\label{ineq:TL1_hat_a}
        \begin{aligned}
            \text{TL1}_a(\hat{A}) &= \text{TL1}_a(A_0 + \hat{A} - A_0) =\text{TL1}_a (A_0 + P_{A_0}^\perp (\hat{A} - A_0) + P_{A_0} (\hat{A} - A_0)) \\
            &=  \text{TL1}_a (A_0 + P_{A_0}^\perp (\hat{A} - A_0)) + \langle \nabla \text{TL1}_a (\Tilde{A}), P_{A_0} (\hat{A} - A_0) \rangle\\
            &\geq \text{TL1}_a (A_0 )+  \text{TL1}_a(P_{A_0}^\perp (\hat{A} - A_0)) - \left\|\nabla \text{TL1}_a (\Tilde{A}) \right\| \|P_{A_0} (\hat{A} - A_0)\|_* \\
            &\geq \text{TL1}_a (A_0 )+  \text{TL1}_a(P_{A_0}^\perp (\hat{A} - A_0)) - \frac{a(1+a)}{a^2} \|P_{A_0} (\hat{A} -A_0)\|_*,
        \end{aligned}
    \end{equation}
    where we use the penultimate inequality because of Lemma \ref{lemma4}, the duality of the nuclear norm, and the last inequality holds because of  \eqref{eq:derivative}.

    Then, we have 
    \begin{equation}
        \begin{aligned}
            &\frac{1}{\nu m_1m_2}\|\hat{A} - A_0\|_F^2 \leq \|\hat{A} - A_0\|^2_{L_2(\Pi)} \\
            &\lesssim \frac{1}{n}\sum_{i=1}^{n} \langle T_i, \hat{A} -A_0 \rangle^2 + \zeta^2 \sqrt{\frac{L\log d}{n}} + \zeta \frac{\|\hat{A} - A_0\|_*}{\sqrt{m_1m_2}} \sqrt{\frac{LM\log d}{n}}\\
            &\lesssim 2\|\Sigma\|\|\hat{A} - A_0\|_* +\lambda \text{TL1}_a (A_0) - \lambda \text{TL1}_a (\hat{A}) + \zeta^2 \sqrt{\frac{L\log d}{n}} + \zeta \frac{\|\hat{A} - A_0\|_*}{\sqrt{m_1m_2}} \sqrt{\frac{LM\log d}{n}}\\
            &\lesssim \left\{ 2\|\Sigma\| + \frac{\zeta}{\sqrt{m_1m_2}} \sqrt{\frac{LM\log d}{n}} \right\} \|\hat{A} - A_0\|_* + \lambda \text{TL1}_a (A_0) - \lambda \text{TL1}_a (\hat{A}) + \zeta^2 \sqrt{\frac{L\log d}{n}}\\
            &\lesssim \frac{(\zeta \vee \sigma)}{\sqrt{m_1m_2}} \sqrt{\frac{Ld\log d}{n}} \big(\|P_{A_0}^\perp (\hat{A} - A_0)\|_* + \|P_{A_0} (\hat{A} - A_0)\|_* \big) + \lambda \text{TL1}_a (A_0) + \zeta^2 \sqrt{\frac{L\log d}{n}}\\
            &\qquad - \lambda \left\{ \text{TL1}_a (A_0 )+  \text{TL1}_a(P_{A_0}^\perp (\hat{A} - A_0)) - \frac{a(1+a)}{a^2} \|P_{A_0} (\hat{A} -A_0)\|_* \right\}\\
            &\lesssim \left\{ \frac{(\zeta \vee \sigma)}{\sqrt{m_1m_2}} \sqrt{\frac{Ld\log d}{n}} + \lambda \frac{a(1+a)}{a^2} \right\} \|P_{A_0} (\hat{A} -A_0)\|_* + \zeta^2 \sqrt{\frac{L\log d}{n}} \\
            &\qquad + \left\{ \frac{(\zeta \vee \sigma)}{\sqrt{m_1m_2}} \sqrt{\frac{Ld\log d}{n}} - \lambda \frac{1+a}{a+\zeta \sqrt{m_1m_2}}\right\} \|P_{A_0}^\perp (\hat{A} -A_0)\|_*,
        \end{aligned}
    \end{equation}
    where the  second last inequality is obtained by \eqref{ineq:TL1_hat_a}, and the last inequality is achieved by $\text{TL1}_a(A) \geq \frac{1+a}{a+\sigma_1(A)}$ and $\sigma_1(A) \leq \zeta \sqrt{m_1m_2}$.
    
    Taking $\lambda \asymp \frac{a+\zeta \sqrt{m_1m_2}}{1+a} \frac{(\zeta \vee \sigma)}{\sqrt{m_1m_2}} \sqrt{\frac{Ld\log d}{n}}$, we get
    \begin{equation*}
        \begin{aligned}
            \lambda \frac{a(1+a)}{a^2} \asymp \frac{a(1+a)}{a^2} \frac{a+\zeta \sqrt{m_1m_2}}{1+a} \frac{(\zeta \vee \sigma)}{\sqrt{m_1m_2}} \sqrt{\frac{Ld\log d}{n}}= \frac{a+\zeta \sqrt{m_1m_2}}{a}\frac{(\zeta \vee \sigma)}{\sqrt{m_1m_2}} \sqrt{\frac{Ld\log d}{n}}.
        \end{aligned}
    \end{equation*}
    Together with $a^{-1} = \mathcal{O}((\zeta \sqrt{m_1m_2})^{-1})$, it deduces to
    \begin{equation}
        \begin{aligned}
            \frac{1}{\nu m_1m_2}\|\hat{A} - A_0\|_F^2 &\lesssim \frac{(\zeta \vee \sigma)}{\sqrt{m_1m_2}} \sqrt{\frac{Ld\log d}{n}} \sqrt{2\tau} \|\hat{A} - A_0\|_F + \zeta^2 \sqrt{\frac{L\log d}{n}},
        \end{aligned}
    \end{equation}
    because of 
    \begin{equation*}
      %  \begin{aligned}
            \|P_{A_0} (\hat{A} -A_0)\|_* \leq \sqrt{\rank (P_{A_0} (\hat{A} -A_0))} \|\hat{A} - A_0\|_\infty \leq \sqrt{2\;\rank (\hat{A} - A_0)} \|\hat{A} - A_0\|_F 
        \leq  \sqrt{2\tau} \|\hat{A} - A_0\|_F.
     %   \end{aligned}
    \end{equation*}
    Therefore, we have
    \begin{equation*}
        \begin{aligned}
            \frac{1}{m_1m_2}\|\hat{A} - A_0\|_F^2 &\lesssim \nu \frac{(\zeta \vee \sigma)}{\sqrt{m_1m_2}} \sqrt{\frac{Ld\log d}{n}} \sqrt{2\tau} \|\hat{A} - A_0\|_F + \nu \zeta^2 \sqrt{\frac{L\log d}{n}},\\
            \frac{1}{m_1m_2}\|\hat{A} - A_0\|_F^2 &\lesssim \nu^2(\zeta^2 \vee \sigma^2) \tau \frac{Ld\log d}{n} + \nu \zeta^2 \sqrt{\frac{L\log d}{n}}.
        \end{aligned}
    \end{equation*}
  In summary, there exist two constants $C_3$ and $C_4$ only depending on $c_0$ such that the inequality
    \begin{equation}
        \begin{aligned}
            \frac{1}{m_1m_2}\|\hat{A} - A_0\|_F^2 &\leq C_3 \nu^2(\zeta^2 \vee \sigma^2) \tau \frac{Ld\log d}{n} + C_4 \nu \zeta^2 \sqrt{\frac{L\log d}{n}},
        \end{aligned}
    \end{equation}
    holds with probability at least $1-\frac{\kappa+1}{d}$.
\end{proof}

\section{Proof of Theorem \ref{Them4} and Corollary \ref{corollary1}}
\label{A3:Them4,Cor1}

\begin{lemma}\label{lemma5}
    If $\lambda^{-1} = \mathcal{O} \left( \left( \frac{(\zeta \vee \sigma)}{\sqrt{m_1m_2}} \frac{a + \zeta\sqrt{m_1m_2}}{1+a} \sqrt{\frac{Ld\log d}{n}} \right)^{-1} \right)$, for any $a = \smallO((m_1m_2)^\frac{1}{4})$, there exists a constant $c_1$ such that the smallest non-zero singular value of 
  the estimator $\hat{A} \in \mathbb{R}^{m_1 \times m_2}$ obtained by \eqref{est:hatA}  shall be greater than or equal to $c_1 \left( a^3 + a^2 \sqrt{m_1m_2} \right)^\frac{1}{4}$.
\end{lemma}
\begin{proof}[Proof of Lemma \ref{lemma5}]
   Suppose the estimator matrix $\hat{A}$ is of rank $k$ and the true matrix $A_0$ is of rank $r_0.$ We ignore the oracle case when $k\leq r_0$, and instead assume that $r_0 \leq k \leq m$. Denote the smallest non-zero singular value of $\hat{A}$ by $\sigma_k$. We prove by contradiction by assuming   
   $\sigma_k\leq c_1 \left( a^3 + a^2 \sqrt{m_1m_2} \right)^\frac{1}{4}$.

    Suppose that the left and right orthonormal singular vectors of $A$ are $\{u_j\}$ and $\{v_j\}$, while the diagonal matrix $S = \mbox{diag}(\sigma_1, \dots, \sigma_m)$ contains the singular values of $A$ in decreasing order. Then we have $A = \sum_{j=1}^{m} \sigma_j u_j v_j^\intercal$ by SVD. Similarly, we write $A_0 = \sum_{j=1}^{m} \sigma_j^* u_j^* v_j^*{^\intercal}$, where $\{u^*_j\}$, $\{v_j^*\}$, and $S^* = \mbox{diag}(\sigma_1^*, \dots, \sigma_m^*)$ are the left, right orthonormal singular vectors and the singular values of $A_0$, respectively.
           Denote $Q(A)$ to be the objective function of \eqref{est:hatA}, i.e., $$
    Q(A) = \frac{1}{n}\sum_{i=1}^{n} \left( \langle T_i, A \rangle -Y_i \right )^2 + \lambda \mathrm{TL}1_a(A) := l(A) + \lambda\mathrm{TL}1_a(A),$$
    where $l(A)$ denotes the least square term that can be expressed as 
    $$l(A) = \frac{1}{n}\sum_{i=1}^{n} \left( \langle T_i, \sum_{j=1}^{m} \sigma_j u_j v_j^\intercal \rangle -Y_i \right )^2.$$

    With fixed $U$ and $V$, we can compute the derivative of $Q$ with any singular value $\sigma_s$ with $s > k$ %\yl{[below apply to A?]}
    \begin{equation}
    \begin{aligned}
        \frac{\partial Q(S)}{\partial \sigma_s} & = \frac{\partial l(S)}{\partial \sigma_s} + \lambda \frac{\partial \mathrm{TL}1_a(S)}{\partial \sigma_s}\\
        &= \frac{2}{n} \sum_{i=1}^{n} \left[ \langle T_i, \sum_{j=1}^{m} \sigma_j u_j v_j^\intercal \rangle -Y_i \right] \langle T_i, u_s v_s^\intercal \rangle + \lambda \frac{a(1+a)}{(a+\sigma_s)^2}\\
        &=\frac{2}{n} \sum_{i=1}^{n} \left[ \langle T_i, \sum_{j=1}^{m} \sigma_j^* u_j^* v_j^*{^\intercal} \rangle -Y_i \right] \langle T_i, u_s v_s^\intercal \rangle \\
        & \qquad + \frac{2}{n} \sum_{i=1}^{n} \langle T_i, \sum_{j=1}^{m} \left(\sigma_j u_j v_j{^\intercal} - 
        \sigma_j^* u_j^* v_j^*{^\intercal} \right)\rangle \langle T_i, u_s v_s^\intercal \rangle + \lambda \frac{a(1+a)}{(a+\sigma_s)^2}\\
        &= \frac{2}{n} \sum_{i=1}^{n} \left[ \langle T_i, A_0 \rangle -Y_i \right] \langle T_i, u_s v_s^\intercal \rangle  + \frac{2}{n} \sum_{i=1}^{n} \langle T_i, \hat{A}-A_0 \rangle \langle T_i, u_s v_s^\intercal \rangle + \lambda \frac{a(1+a)}{(a+\sigma_s)^2}\\
        &= 2\langle \Sigma, u_s v_s^\intercal \rangle + \frac{2}{n} \sum_{i=1}^{n} \langle T_i, \hat{A}-A_0 \rangle \langle T_i, u_s v_s^\intercal \rangle + \lambda \frac{a(1+a)}{(a+\sigma_s)^2}  \\
        &\leq 2\|\Sigma\| \|u_s v_s^\intercal\|_* + 2 \sqrt{\frac{1}{n}\sum_{i=1}^{n} \langle T_i, \hat{A}-A_0 \rangle^2} \sqrt{\frac{1}{n} \sum_{i=1}^{n} \langle T_i, u_s v_s^\intercal \rangle^2} + \lambda \frac{a(1+a)}{(a+\sigma_s)^2}\\
        &=: I_1 + I_2+ I_3,
    \end{aligned}
    \end{equation}
    where the last-second inequality is obtained by the duality of the nuclear norm and Cauchy-Schwarz inequality. We estimate the three terms $I_1, I_2,$ and $I_3$ individually. 
    
\textbf{For $I_1$.} Since $u_s v_s^\intercal$ is a rank-1 matrix, then $\|u_s v_s^\intercal\|_* = \|u_s\| \|v_s^\intercal\| = 1$. We further use  \eqref{Bernstein} to get
    \begin{equation}\label{I_1}
        I_1 = \mathcal{O} (\sqrt{\frac{Ld \log d}{nm_1m_2}}),
    \end{equation}
    holds with probability at least $1-\frac{1}{d}$.

    \textbf{For $I_2$.}
    %By Lemma \ref{lemma2}, we have
    Following Lemma \ref{lemma1} and Lemma \ref{lemma2}, we take $\gamma = 1/\sqrt{m_1m_2}$ in the constraint set \eqref{constrain_set}  and let $H_i = \langle T_i, u_s v_s^\intercal \rangle^2$ and define
    \begin{equation*}
       % \begin{aligned}
            V := n \mathbb{E}(H_i^2) + 16\frac{n}{\sqrt{m_1m_2}}\sqrt{\frac{Ld\log d}{nm_1m_2}} = \frac{n}{m_1m_2} + 16\frac{n}{\sqrt{m_1m_2}} \sqrt{\frac{Ld\log d}{nm_1m_2}}  \lesssim \frac{n}{m_1m_2}, 
       % \end{aligned}
    \end{equation*}
    by $n\gtrsim d\log d$.
Taking  $t \gtrsim \sqrt{(Ld\log d) / (nm_1m_2)}$ and using the Talagrand concentration inequality  \citep[Theorem 2.6]{koltchinskii2011oracle}, 
we have for a universal constant $K$ that
    \begin{equation} \label{inq:talagrand}
    %\begin{aligned}
            \mathbb{P}\bigg\{ \underset{u_s v_s^\intercal \in \mathcal{K}(1, \gamma)}{\sup}|Z_\gamma - \mathbb{E}(Z_\gamma)| \geq t \bigg\} \leq K \exp \bigg\{ -\frac{1}{K} tn \; \log \left(1+tm_1m_2\right) \bigg\}\leq K \exp \bigg\{ -\frac{1}{K} \frac{t^2 nm_1m_2}{1+tm_1m_2} \bigg\},
   % \end{aligned}
    \end{equation}
which implies that the inequality
    \begin{equation}
        \begin{aligned}
            \frac{1}{n}\langle T_i, u_s v_s^\intercal \rangle^2 & \leq \|u_s v_s^\intercal\|^2_{L_2(\Pi)} + \sqrt{\frac{Ld\log d}{nm_1m_2}}
            \leq \frac{\nu}{m_1m_2} \|u_s v_s^\intercal\|^2_F + \sqrt{\frac{Ld\log d}{nm_1m_2}}\\
            &= \frac{\nu}{m_1m_2} + \sqrt{\frac{Ld\log d}{nm_1m_2}}
             = \bigO(\sqrt{\frac{Ld\log d}{nm_1m_2}}),
        \end{aligned}
    \end{equation}
 holds   with high probability. By Lemma \ref{lemma3} with $\lambda^{-1} = \mathcal{O} \left(\left( \frac{(\zeta \vee \sigma)}{\sqrt{m_1m_2}} \frac{a + \zeta\sqrt{m_1m_2}}{1+a} \sqrt{\frac{Ld\log d}{n}} \right)^{-1}\right)$, we get
    \begin{equation}
        \begin{aligned}
            \frac{1}{n}\langle T_i, \hat{A}-A_0 \rangle^2 &\leq \|\hat{A}-A_0\|^2_{L_2(\Pi)} + \zeta^2 \sqrt{\frac{L\log d}{n}} + \zeta \frac{\|\hat{A}-A_0\|_*}{\sqrt{m_1m_2}} \sqrt{\frac{LM\log d}{n}}\\
            &\leq \frac{\nu}{m_1m_2} \|\hat{A}-A_0\|_F^2 + \zeta^2 \sqrt{\frac{L\log d}{n}} + \zeta r_0 \sqrt{\frac{Ld\log d}{n}}\\
            & \lesssim \nu^2 \left\{ (\zeta \vee \sigma) \sqrt{\frac{Ld\log d}{n}}  \frac{\|A_0\|_{*}}{\sqrt{m_1m_2}} + \lambda \mathrm{TL}1_a(A_0) \right\} + \nu^2 \zeta^2 \sqrt{\frac{L\log d}{n}}\\
            &= \mathcal{O}\left(\lambda (1+a)\right),
        \end{aligned}
    \end{equation}
    where the last inequality is because of $\mathrm{TL}1_a(A_0) \leq (1+a)r_0$.
    Therefore, $I_2 = \mathcal{O}\left( \sqrt{\lambda(1+a)} \left(\frac{Ld\log d}{nm_1m_2}\right)^{1/4}\right)$.
    Together with  \eqref{I_1}, we have
    \begin{equation}
        \begin{aligned}
            I_1+I_2 = \mathcal{O}\left( \sqrt{\lambda(1+a)} \left(\frac{Ld\log d}{nm_1m_2}\right)^{1/4}\right).
        \end{aligned}
    \end{equation}

   \textbf{For $I_3$.} We know that $I_3 = \lambda \frac{a(1+a)}{(a+\sigma_s)^2}$ for $s > k$. We verify whether $I_3 > I_1 + I_2$, when $a = \smallO \left((m_1m_2)^\frac{1}{4}\right)$ and $\sigma_s \leq c_1 \left( a^3 + a^2 \sqrt{m_1m_2} \right)^\frac{1}{4}$. Then we have
        \begin{align}
            \frac{I_3}{I_1+I_2} &\gtrsim \frac{a\sqrt{\lambda(1+a)}}{(a+\sigma_s)^2} \left( \frac{Ld\log d}{nm_1m_2}\right)^{-\frac{1}{4}}\\
            &\gtrsim \frac{a \sqrt{1+a}}{(a+\sigma_s)^2} \sqrt{\frac{a + \zeta\sqrt{m_1m_2}}{\sqrt{m_1m_2}(1+a)} \sqrt{\frac{Ld\log d}{n}} } \left( \frac{Ld\log d}{nm_1m_2}\right)^{-\frac{1}{4}}\\
            &=\frac{a\sqrt{(a+\zeta\sqrt{m_1m_2})}}{(a+\sigma_s)^2},
        \end{align}
        which leads to the desired bound,
        \begin{equation}
                     \left( \frac{I_3}{I_1+I_2} \right)^2 \gtrsim \frac{a^2 (a+\zeta\sqrt{m_1m_2})}{(a+\sigma_s)^4} > \frac{a^2 (a+\zeta\sqrt{m_1m_2})}{\sigma_s^4} \gtrsim \frac{a+\zeta\sqrt{m_1m_2}}{a+\sqrt{m_1m_2}} = \mathcal{O}(1). 
        \end{equation}

    Overall, we have $\frac{\partial Q(S)}{\partial \sigma_s} > 0$ is always satisfied, then  it follows from \citet[Lemma 1]{fan2001variable} that there exists a constant $c_1$ such that $\sigma_k \geq c_1 \left( a^3 + a^2 \sqrt{m_1m_2} \right)^\frac{1}{4}$.
\end{proof}

\begin{proof}[Proof of Theorem 4]
    %By the optimality of the estimator $\hat{A}$, we have
    \begin{comment}
        \begin{align*}
        \frac{1}{n} \sum_{i=1}^{n} (Y_i - \langle T_i, \hat{A} \rangle)^2 + \lambda \text{TL1}_a (\hat{A}) \leq \frac{1}{n} \sum_{i=1}^{n} (Y_i - \langle T_i, A_0 \rangle)^2 + \lambda \text{TL1}_a (A_0),
    \end{align*}
    Then, it follows
    \begin{equation}
        \begin{aligned}
        \lambda \text{TL1}_a (\hat{A}) &\leq \frac{1}{n} \sum_{i=1}^{n} (Y_i - \langle T_i, A_0 \rangle)^2 + \lambda \text{TL1}_a (A_0)\\
        \text{TL1}_a (\hat{A}) &\leq \frac{\sigma^2}{n \lambda} \sum_{i=1}^{n} \xi_i^2 + \text{TL1}_a(A_0) \leq \frac{\sigma^2}{\lambda} + (1+a)\tau,  
        \end{aligned}
    \end{equation}
    \end{comment}
    By the optimality of the estimator $\hat{A}$ together with \eqref{inq:base_derived}, we deduce
    \begin{align*} 
        \lambda \text{TL1}_a (\hat{A}) &\leq 2\|\Sigma\| \sqrt{\rank (\hat{A}- A_0)} \|\hat{A}- A_0\|_F +\lambda \text{TL1}_a (A_0).
\end{align*}
We further use $\mathrm{TL}1_a (A_0) \leq (1+a) \rank (A_0)$ to get
        \begin{align*}
        \lambda \text{TL1}_a (\hat{A}) \lesssim \sqrt{\frac{Ld\log d}{nm_1m_2}} \sqrt{\rank(\hat{A}) + \rank(A_0)} \sqrt{m_1m_2\lambda \frac{(1+a)\zeta \sqrt{m_1m_2}}{a+\zeta \sqrt{m_1m_2}}}+ \lambda(1+a)\rank(A_0).
        \end{align*}
Using $\rank(\hat{A}) \geq \rank(A_0)$, we have
        \begin{align*}
        \text{TL1}_a (\hat{A}) \lesssim \lambda^{-\frac{1}{2}} \sqrt{\rank(\hat{A})}\sqrt{\frac{Ld\log d}{n}}  \sqrt{\frac{(1+a)\zeta \sqrt{m_1m_2}}{a+\zeta \sqrt{m_1m_2}}}+ (1+a)\rank(A_0).
    \end{align*}

    \begin{comment}
        \begin{align*} 
        \lambda \text{TL1}_a (\hat{A}) &\leq 2\|\Sigma\| \sqrt{\rank (\hat{A}- A_0)} \|\hat{A}- A_0\|_F +\lambda \text{TL1}_a (A_0)\\
        \lambda \text{TL1}_a (\hat{A}) &\lesssim \sqrt{\frac{Ld\log d}{nm_1m_2}} \sqrt{\rank(\hat{A}) + \rank(A_0)} \sqrt{m_1m_2\lambda \frac{(1+a)\zeta \sqrt{m_1m_2}}{a+\zeta \sqrt{m_1m_2}}}\\
        &\qquad + \lambda(1+a)\rank(A_0), \qquad \qquad \qquad \left(\text{by $\mathrm{TL}1_a (A_0) \leq (1+a) \rank (A_0)$}\right)\\
        \text{TL1}_a (\hat{A}) &\lesssim \lambda^{-\frac{1}{2}}\sqrt{\frac{Ld\log d}{n}} \sqrt{\rank(\hat{A})} \sqrt{\frac{(1+a)\zeta \sqrt{m_1m_2}}{a+\zeta \sqrt{m_1m_2}}}\\
        &\qquad + (1+a)\rank(A_0), \qquad \qquad \qquad \qquad \qquad \left(\text{by $\rank(\hat{A}) \geq \rank(A_0)$}\right)\\
        &\lesssim \sqrt{\rank(\hat{A})} \sqrt{\frac{(1+a)\sqrt{m_1m_2}}{a+ \zeta\sqrt{m_1m_2}} \sqrt{\frac{n}{Ld\log d}}} \sqrt{\frac{Ld\log d}{n}} \sqrt{\frac{(1+a)\zeta \sqrt{m_1m_2}}{a+\zeta \sqrt{m_1m_2}}}\\
        &\qquad + (1+a)\rank(A_0) \qquad \qquad \qquad \qquad \qquad \quad \left( \text{replace $\lambda$ with its order} \right)\\
        &\lesssim \sqrt{\rank(\hat{A})} \left(\frac{Ld\log d}{n}\right)^{\frac{1}{4}} \frac{(1+a)\sqrt{m_1m_2}}{a+ \sqrt{m_1m_2}} + (1+a)\rank(A_0)
    \end{align*}
    \end{comment}
    
    %here, the last inequality is satisfied by the definition of the trace regression model, $\xi_i$ are independent noise and $\mathbb{E}(\xi_i) = 1$, and $\mathrm{TL}1_a (A_0) \leq (1+a) \rank (A_0)$.

   Since $\mathrm{TL}1_a(\hat{A}) \geq \rank (\hat{A}) \frac{(1+a) \sigma_{\rank (\hat{A})}(\hat{A})}{a+\sigma_{\rank (\hat{A})}(\hat{A})}$, we have
    \begin{comment}
        \begin{equation}
        \begin{aligned}
            \rank (\hat{A}) &\leq \left [\frac{\sigma^2}{\lambda} + (1+a)\tau \right] \frac{a+\sigma_{\rank (\hat{A})}(\hat{A})}{(1+a) \sigma_{\rank (\hat{A})}(\hat{A})}\\
            %&\lesssim \frac{1}{\lambda} \frac{a+\sigma_{\rank (\hat{A})}(\hat{A})}{(1+a) \sigma_{\rank (\hat{A})}(\hat{A})} + \frac{r_0 \left(a+\sigma_{\rank (\hat{A})}(\hat{A}) \right)}{ \sigma_{\rank (\hat{A})}(\hat{A})}
        \end{aligned}
    \end{equation}
    \end{comment}

    \begin{align*}
        \rank (\hat{A}) &\lesssim \left( \lambda^{-\frac{1}{2}} \sqrt{\rank(\hat{A})}\sqrt{\frac{Ld\log d}{n}}  \sqrt{\frac{(1+a)\zeta \sqrt{m_1m_2}}{a+\zeta \sqrt{m_1m_2}}} + (1+a)\rank(A_0) \right) \frac{a+\sigma_{\rank (\hat{A})}}{(1+a)\sigma_{\rank (\hat{A})}},
        %&= \left( \sqrt{\rank(\hat{A})} \left(\frac{Ld\log d}{n}\right)^{\frac{1}{4}} \frac{\sqrt{m_1m_2}}{a+ \sqrt{m_1m_2}} + \rank(A_0) \right) \frac{a+\sigma_{\rank (\hat{A})}}{\sigma_{\rank (\hat{A})}}\\
    \end{align*}
    \begin{comment}
        \begin{align*}
        \rank (\hat{A}) &\lesssim \left( \sqrt{\rank(\hat{A})} \left(\frac{Ld\log d}{n}\right)^{\frac{1}{4}} \frac{(1+a)\sqrt{m_1m_2}}{a+ \sqrt{m_1m_2}} + (1+a)\rank(A_0) \right) \frac{a+\sigma_{\rank (\hat{A})}}{(1+a)\sigma_{\rank (\hat{A})}}\\
        &= \left( \sqrt{\rank(\hat{A})} \left(\frac{Ld\log d}{n}\right)^{\frac{1}{4}} \frac{\sqrt{m_1m_2}}{a+ \sqrt{m_1m_2}} + \rank(A_0) \right) \frac{a+\sigma_{\rank (\hat{A})}}{\sigma_{\rank (\hat{A})}}\\
    \end{align*}
    \end{comment}
which can be derived into two scenarios:

Scenario (i):
    \begin{align}
        \rank (\hat{A}) & \lesssim \lambda^{-1} \frac{Ld\log d}{n} \frac{\sqrt{m_1m_2}}{(1+a)(a+ \sqrt{m_1m_2})}  \left(\frac{a+\sigma_{\rank (\hat{A})}}{\sigma_{\rank (\hat{A})}}\right)^2;
    \end{align}
and Scenario (ii): 
    \begin{align}
        \rank (\hat{A}) & \lesssim \rank(A_0) \frac{a+\sigma_{\rank (\hat{A})}}{\sigma_{\rank (\hat{A})}}.
    \end{align}
    
    Using the condition that $\sigma_{\rank (\hat{A})} (\hat{A}) \geq c_1 \left( a^3 + a^2 \sqrt{m_1m_2} \right)^\frac{1}{4}$ by Lemma \ref{lemma5}, these two scenarios can be further simplified to
    \begin{comment}
        \begin{equation}
        \begin{aligned}
            \rank (\hat{A}) &\lesssim \left (\frac{1}{\lambda} + a + 1 \right) \frac{a + \left( a^3 + a^2 \sqrt{m_1m_2} \right)^\frac{1}{4}}{(1+a) \left( a^3 + a^2 \sqrt{m_1m_2} \right)^\frac{1}{4}}\\
            &= \left (\frac{1}{\lambda} + a + 1 \right) \frac{\sqrt{a} + \left( a + \sqrt{m_1m_2} \right)^\frac{1}{4}}{(1+a) \left( a + \sqrt{m_1m_2} \right)^\frac{1}{4}}.
        \end{aligned}
    \end{equation}
    \end{comment}

Scenario (i):
    \begin{equation} \label{senario:1}
    \begin{aligned}
        \rank(\hat{A}) &\lesssim \lambda^{-1} \frac{Ld\log d}{n} \frac{\sqrt{m_1m_2}}{(1+a)(a+ \sqrt{m_1m_2})}   \left(\frac{a+ \left( a^3 + a^2 \sqrt{m_1m_2} \right)^\frac{1}{4}}{\left( a^3 + a^2 \sqrt{m_1m_2} \right)^\frac{1}{4}}\right)^2\\
        &\lesssim \lambda^{-1} \frac{Ld\log d}{n} \frac{\sqrt{m_1m_2}}{(1+a)(a+ \sqrt{m_1m_2})}  \left(\frac{\sqrt{a}}{\left( a + \sqrt{m_1m_2} \right)^\frac{1}{4}} +1 \right)^2.
        &    
    \end{aligned}
    \end{equation}
    
Scenario (ii):
    \begin{equation} \label{senario:2}
    \begin{aligned}
        \rank (\hat{A}) & \lesssim \rank(A_0) \frac{a+\left( a^3 + a^2 \sqrt{m_1m_2} \right)^\frac{1}{4}}{\left( a^3 + a^2 \sqrt{m_1m_2} \right)^\frac{1}{4}} = \rank(A_0) \left(\frac{\sqrt{a}}{\left( a + \sqrt{m_1m_2} \right)^\frac{1}{4}} +1 \right).
    \end{aligned}
    \end{equation}
    
    Combine these two senarios \eqref{senario:1} and \eqref{senario:2}, we have
    %\yl{[which two bounds]}, 
    \begin{align*}
    \rank (\hat{A}) \lesssim & \max \left\{ \lambda^{-1} \frac{Ld \log d}{n} \frac{\sqrt{m_1m_2}}{(1+a)(a+ \sqrt{m_1m_2})}  \left(\frac{\sqrt{a}}{\left( a + \sqrt{m_1 m_2} \right)^{1/4}} + 1 \right)^2, \right. \\
    & \qquad \qquad \left. \rank(A_0) \left(\frac{\sqrt{a}}{\left( a + \sqrt{m_1 m_2} \right)^{1/4}} + 1 \right) \right\}.
    \end{align*}

    Hence, there exists a constant $C_7$ only depending on $c_0$ such that 
    \begin{align*}
    \rank (\hat{A}) \leq & C_7 \left\{ \lambda^{-1} \frac{Ld \log d}{n} \frac{\sqrt{m_1m_2}}{(1+a)(a+ \sqrt{m_1m_2})}  \left(\frac{\sqrt{a}}{\left( a + \sqrt{m_1 m_2} \right)^{1/4}} + 1 \right)^2 \right. \\
    & \qquad \qquad \qquad \qquad \qquad \qquad \qquad \left. + \rank(A_0) \left(\frac{\sqrt{a}}{\left( a + \sqrt{m_1 m_2} \right)^{1/4}} + 1 \right) \right\},
    \end{align*}
    with high probability.
\end{proof}

\begin{proof}[Proof of Corollary \ref{corollary1}]
  Replacing $\lambda$ with the order of $\frac{(\zeta \vee \sigma)}{\sqrt{m_1m_2}} \frac{a + \zeta \sqrt{m_1m_2}}{1+a} \sqrt{\frac{Ld \log d}{n}}$ in Theorem \ref{Them3}, we get
    \begin{equation}
        \begin{aligned}
            \frac{1}{m_1m_2}\|\hat{A} - A_0\|_F^2 &\lesssim \nu \tau \frac{(\zeta \vee \sigma)}{\sqrt{m_1m_2}} \frac{a + \zeta \sqrt{m_1m_2}}{1+a} \sqrt{\frac{Ld \log d}{n}} \frac{(1+a)\zeta \sqrt{m_1m_2}}{a+\zeta \sqrt{m_1m_2}} + \nu \zeta^2 \sqrt{\frac{L\log d}{n}} \\
            &= \nu \tau (\zeta^2 \vee \sigma^2) \sqrt{\frac{Ld \log d}{n}} + \nu \zeta^2 \sqrt{\frac{L\log d}{n}}.
        \end{aligned}
    \end{equation}
 According to Theorem \ref{Them4} regarding the upper bound on the rank of the estimator $\hat{A}$ (i.e., two scenarios)  and $a = \smallO\left((m_1m_2)^{1/4}\right)$, we have
 
Scenario (i):
    \begin{equation} \label{cor_s:1}
    \begin{aligned}
        \rank (\hat{A}) &\lesssim \frac{(1+a)\sqrt{m_1m_2}}{a+\sqrt{m_1m_2}} \sqrt{\frac{n}{Ld\log d}}\frac{Ld\log d}{n} \frac{\sqrt{m_1m_2}}{(1+a)(a+ \sqrt{m_1m_2})}  \left(\frac{\sqrt{a}}{\left( a + \sqrt{m_1m_2} \right)^\frac{1}{4}} +1 \right)^2\\
        &\lesssim \frac{m_1m_2}{(a+\sqrt{m_1m_2})^2} \sqrt{\frac{Ld\log d}{n}} \left(\frac{\sqrt{a}}{\left( a + \sqrt{m_1m_2} \right)^\frac{1}{4}} +1 \right)^2= \mathcal{O}_p(1),
    \end{aligned}
    \end{equation}
Scenario (ii): 
    \begin{equation}\label{cor_s:2}
        \rank (\hat{A}) = \mathcal{O}_p(\rank(A_0)).
    \end{equation}

    \begin{comment}
        \begin{equation}
        \begin{aligned}
            \rank (\hat{A}) &= \mathcal{O}_p \left( \left(\frac{1+a}{a+\zeta \sqrt{m_1m_2}} \sqrt{\frac{nm_1m_2}{Ld \log d}} +a+1\right) \frac{\sqrt{a} + \left( a + \sqrt{m_1m_2} \right)^\frac{1}{4}}{ (1+a) \left( a + \sqrt{m_1m_2} \right)^\frac{1}{4}}\right)\\
            &= \mathcal{O}_p \left( \left(\frac{1}{a+\zeta \sqrt{m_1m_2}} \sqrt{\frac{nm_1m_2}{Ld \log d}} +1\right) \frac{\sqrt{a} + \left( a + \sqrt{m_1m_2} \right)^\frac{1}{4}}{ \left( a + \sqrt{m_1m_2} \right)^\frac{1}{4}}\right)\\
            &= \mathcal{O}_p \left(\sqrt{\frac{n}{Ld\log d}} +1\right),  \quad \text{by $a = \smallO \left((m_1m_2)^\frac{1}{4}\right)$}\\
            &= \mathcal{O}_p \left( \sqrt{\frac{n}{Ld\log d}} \right).
        \end{aligned}
    \end{equation}
    \end{comment}
   Using \eqref{cor_s:1} and \eqref{cor_s:2}, we obtain the desired result $\rank (\hat{A}) = \mathcal{O}_p(\rank(A_0))$.   
%    Therefore, we obtained the desired results.
\end{proof}

\section{Parameter tuning}
\label{A6:tune_para}

We compare the TL1 model \eqref{est:hatA} with nuclear-norm \citep{candes2012exact}, max-norm \citep{cai2016matrix}, and a hybrid approach in a combination of the nuclear norm and max norm \citep{fang2018max}. Each method has its respective hyper-parameters. We conduct a grid search to find the optimal parameters that yield the smallest relative error (RE) defined in Section \ref{sec:4.1} within the following ranges: 
%\yl{[what do you mean by the last sentence? we may have to move the respective set in supp due to spacing]}
\begin{itemize}
    \item For max-norm: $\lambda \in \{10^{-1}, 10^{-2}, 10^{-3}, 10^{-4}, 10^{-5}, 10^{-6}\}$
    \item For hybrid: 
        \begin{itemize}
            \item $\lambda \in \{10^{-1}, 10^{-2}, 10^{-3}, 10^{-4}, 10^{-5}, 10^{-6}\}$
            \item $\mu \in \{10^{-1}, 10^{-2}, 10^{-3}, 10^{-4}, 10^{-5}, 10^{-6}\}$
        \end{itemize}
    \item For nuclear norm: $\mu \in \{10^{-1}, 10^{-2}, 10^{-3}, 10^{-4}, 10^{-5}, 10^{-6}\}$
    \item For TL1: 
        \begin{itemize}
            \item $a \in \{10^{-1}, 1, 10, 20, 50, 100, 200, 500, 600, 900, 1500, 3000\}$
            \item $\lambda \in \{10^{-1}, 10^{-2}, 10^{-3}, 10^{-4}, 10^{-5}, 10^{-6}\}$
        \end{itemize}
\end{itemize}
All the values for $\lambda$ and $\mu$ are multiplied by $\|Y\|_F$.
Following Algorithm \ref{alg:TL1_ADMM}, we implement the TL1 model by ourselves and fix the parameter $\rho = 0.1$ as it only affects the convergence rate, but not the performance.  We obtain the codes of the hybrid approach \citep{fang2018max} from the authors, which includes the nuclear norm and the max norm as its special case.  All the experiments are run on a MacBook Pro with an Apple M2 chip and 8GB of memory.

\section{More experimental results}
\label{A5:Other_res}

Table \ref{tab_S1} contains the results of relative errors for the reconstructed matrix to the ground truth under Scheme 1, 2, and 3 settings with
noiseless and noisy data for SNR=20. Each reported value is averaged over 50 random realizations with
standard deviation in parentheses. We also highlight the two best RE values in the table in a similar way as Tables \ref{tab:scheme1}-\ref{tab:scheme3}.

\begin{table}[h]
  \centering
  \caption{Similar to Tables \ref{tab:scheme1}-\ref{tab:scheme3} with SNR =20 (mean and standard deviation over 50 trials are reported)}\label{tab_S1} 
\begin{tabular}{ccccccccc}
    \toprule
      \multicolumn{6}{c}{Scheme 1 with SNR=20}\\
    \midrule
       & (r, SR) & Max-norm & Hybrid & Nuclear & TL1\\
    \midrule
      300 &(5, 0.1)& 0.327 (0.012)& \textit{0.064 (0.013)}& 0.152 (0.007)& \textbf{0.029 (0.001)}\\
      &(5, 0.2)& 0.195 (0.005)& \textbf{0.009 (0.002)}& 0.068 (0.001)& \textit{0.013 (0.001)}\\
      &(10, 0.1)& 0.568 (0.011)& \textit{0.495 (0.013)}& 0.505 (0.012)& \textbf{0.102 (0.003)}\\
      &(10, 0.2)& 0.248 (0.005)& \textbf{0.021 (0.002)}& 0.120 (0.002)& \textit{0.024 (0.001)}\\
    \midrule

      500&(5, 0.1)& 0.204 (0.007)& 0.014 (0.000)& \textit{0.013 (0.002)}& \textbf{0.005 (0.000)}\\
      &(5, 0.2)& 0.129 (0.004)& 0.008 (0.000)& \textit{0.007 (0.000)}& \textbf{0.003 (0.000)}\\
      &(10, 0.1)& 0.270 (0.007)& 0.082 (0.006)& \textit{0.081 (0.006)}& \textbf{0.009 (0.001)}\\
      &(10, 0.2)& 0.155 (0.003)& 0.016 (0.000)& \textit{0.011 (0.000)}& \textbf{0.006 (0.000)}\\
    \bottomrule\\

    \toprule
      \multicolumn{6}{c}{Scheme 2 with SNR=20}\\
    \midrule
       & (r, SR) & Max-norm & Hybrid & Nuclear & TL1\\
    \midrule
      300 &(5, 0.1)& 0.305 (0.033)& \textbf{0.281 (0.028)}& 0.759 (0.017)& \textit{0.368 (0.049)}\\
      &(5, 0.2)& 0.174 (0.016)& \textit{0.139 (0.022)}& 0.606 (0.020)& \textbf{0.063 (0.043)}\\
      &(10, 0.1)& 0.480 (0.026)& \textbf{0.477 (0.029)}& 0.798 (0.010)& \textit{0.505 (0.031)}\\
      &(10, 0.2)& 0.217 (0.014)& \textit{0.204 (0.026)}& 0.610 (0.016)& \textbf{0.138 (0.029)}\\
    \midrule

      500&(5, 0.1)& 0.265 (0.027)& \textit{0.209 (0.022)}& 0.753 (0.014)& \textbf{0.048 (0.037)}\\
      &(5, 0.2)& 0.144 (0.007)& \textit{0.126 (0.006)}& 0.606 (0.015)& \textbf{0.006 (0.000)}\\
      &(10, 0.1)& 0.337 (0.015)& \textit{0.286 (0.016)}& 0.762 (0.008)& \textbf{0.130 (0.024)}\\
      &(10, 0.2)& 0.178 (0.007)& \textit{0.145 (0.008)}& 0.609 (0.010)& \textbf{0.012 (0.011)}\\
    \bottomrule\\

    \toprule
      \multicolumn{6}{c}{Scheme 3 with SNR=20}\\
    \midrule
       & (r, SR) & Max-norm & Hybrid & Nuclear & TL1\\
    \midrule
      300 &(5, 0.1)& 0.464 (0.021)& \textit{0.451 (0.023)}& 0.784 (0.010)& \textbf{0.446 (0.031)}\\
      &(5, 0.2)& 0.184 (0.018)& \textit{0.152 (0.024)}& 0.611 (0.017)& \textbf{0.065 (0.035)}\\
      &(10, 0.1)& 0.543 (0.018)& \textit{0.533 (0.018)}& 0.820 (0.007)& \textbf{0.521 (0.018)}\\
      &(10, 0.2)& 0.229 (0.019)& \textit{0.221 (0.022)}& 0.625 (0.012)& \textbf{0.147 (0.034)}\\
    \midrule

      500&(5, 0.1)& 0.393 (0.025)& \textit{0.323 (0.026)}& 0.757 (0.014)& \textbf{0.107 (0.031)}\\
      &(5, 0.2)& 0.148 (0.061)& \textit{0.127 (0.053)}& 0.607 (0.014)& \textbf{0.010 (0.008)}\\
      &(10, 0.1)& 0.508 (0.012)& \textit{0.472 (0.013)}& 0.793 (0.007)& \textbf{0.329 (0.015)}\\
      &(10, 0.2)& 0.182 (0.008)& \textit{0.151 (0.009)}& 0.611 (0.010)& \textbf{0.015 (0.015)}\\
    \bottomrule
    
    \end{tabular}
\end{table}

\newpage
Additionally, we also report the performance on a larger matrix of dimension $1000 \times 1000$ across three schemes with different noise levels as well as on a matrix of dimension $1500 \times 1500$ under three schemes with $(r, SR) = (10, 0.2)$ and SNR = 20. Due to the time constraint, we can only report the results based on one random trial in Table \ref{tab_S2} and Table \ref{tab_1500}. These results of less noisy data and larger dimensions are consistent with what we report in the main text, showing that TL1 outperforms the competing methods.

\begin{table}[ht]
  \centering
  \caption{Similar to Tables \ref{tab:scheme1}-\ref{tab:scheme3}, relative errors of 1000 $\times$ 1000 matrices for one trial}\label{tab_S2}
\begin{tabular}{ccccccccc}
    \hline
      \multicolumn{6}{c}{Scheme 1 without noise}\\
    \hline
       & (r, SR) & Max-norm & Hybrid & Nuclear & TL1\\
    \hline
      1000 &(5, 0.1)& 0.1135& 8.87 $\times10^{-4}$& $\mathit{8.77 \times 10^{-4}}$&  \bm{$1.51\times10^{-5}$}\\
      &(5, 0.2)& 0.0707& $\mathit{1.50\times10^{-4}}$& 5.46$\times10^{-4}$& \bm{$9.50\times10^{-6}$}\\
      &(10,0.1)& 0.1451& ${\mathit6.62\times10^{-4}}$& 1.50$\times10^{-3}$& \bm{$2.60\times10^{-5}$}\\
      &(10,0.2)& 0.0865& $\mathit{1.54\times10^{-4}}$& 8.29$\times10^{-4}$& \bm{$1.45\times10^{-5}$}\\
    \hline\\
    
    \hline
      \multicolumn{6}{c}{Scheme 1 with SNR=10}\\
    \hline
       & (r, SR) & Max-norm & Hybrid & Nuclear & TL1\\
       \hline
      1000&(5, 0.1)& 0.1318& 0.0759& \textit{0.0708}& \textbf{0.0696}\\
      &(5, 0.2)& 0.0811& 0.0689& \textbf{0.0601}& \textit{0.0644}\\
      &(10,0.1)& 0.1508& 0.1081& \textit{0.1043}& \textbf{0.0905}\\
      &(10,0.2)& 0.0970& 0.0781& \textbf{0.0706}& \textit{0.0754}\\
    \hline\\
    
    \hline
      \multicolumn{6}{c}{Scheme 2 without noise}\\
      \hline
       & (r, SR) & Max-norm & Hybrid & Nuclear & TL1\\
    \hline
      1000&(5, 0.1)& 0.2365& \textit{0.1658}& 0.5863& \textbf{0.0129}\\
      &(5, 0.2)& 0.0978& \textit{0.0771}& 0.6069& \textbf{0.0019}\\
      &(10,0.1)& 0.2316& \textit{0.1599}& 0.7476& \textbf{0.0716}\\
      &(10,0.2)& 0.1166& \textit{0.0854}& 0.6062& \textbf{0.0651}\\
    \hline\\
    
    \hline
    \multicolumn{6}{c}{Scheme 2 with SNR=10}\\
    \hline
       & (r, SR) & Max-norm & Hybrid & Nuclear & TL1\\
    
    \hline
      1000&(5, 0.1)& 0.2369& \textit{0.1663}& 0.7490& \textbf{0.0491}\\
      &(5, 0.2)& 0.0982& \textit{0.0769}& 0.6069& \textbf{0.0694}\\
      &(10,0.1)& 0.2321& \textit{0.1609}& 0.7476& \textbf{0.1091}\\
      &(10,0.2)& 0.1167& \textit{0.0858}& 0.6062& \textbf{0.0791}\\
    \hline\\
    
    \hline
      \multicolumn{6}{c}{Scheme 3 without noise}\\
      \hline
       & (r, SR) & Max-norm & Hybrid & Nuclear & TL1\\
    \hline
      1000&(5, 0.1)& 0.3064& \textit{0.2138}& 0.7490& \textbf{0.0293}\\
      &(5, 0.2)& 0.1040& \textit{0.0768}& 0.6069& \textbf{0.0001}\\
      &(10,0.1)& 0.3433& \textit{0.2570}& 0.7491& \textbf{0.1315}\\
      &(10,0.2)& 0.1162& \textit{0.0890}& 0.6062& \textbf{0.0085}\\
    \hline\\
    
    \hline
    \multicolumn{6}{c}{Scheme 3 with SNR=10}\\
    \hline
       & (r, SR) & Max-norm & Hybrid & Nuclear & TL1\\
    \hline
      1000&(5, 0.1)& 0.3064& \textit{0.2138}& 0.7490& \textbf{0.0358}\\
      &(5, 0.2)& 0.1040& \textit{0.0768}& 0.6069& \textbf{0.0133}\\
      &(10,0.1)& 0.3433& \textit{0.2570}& 0.7491& \textbf{0.1595}\\
      &(10,0.2)& 0.1162& \textit{0.0890}& 0.6062& \textbf{0.0693}\\
    \hline
\end{tabular}
\end{table}

\newpage

Table \ref{tab_d2} shows the bias and variance results for estimators obtained from TL1 regularization and nuclear norm regularization under Scheme 2 with SNR=10 over 100 random trials.

\begin{table}
    \caption{Results, with the running time for a single trial in parentheses, for dimension $1500 \times 1500$ under three schemes with (r, SR) = (10, 0.2) and SNR = 20} \label{tab_1500}
    \centering
    \begin{tabular}{cccccc}
        \toprule
        Scheme & Max-norm & Hybrid & Nuclear & TL1 \\
        \midrule
        1 & 0.0625 (2488.8) & \textit{0.0083} (2458.0) & 0.0774 (800.64) & \textbf{0.0027} (866.07) \\
        2 & 0.1048 (2364.8) & \textbf{0.0431} (2350.4) & 0.5972 (753.65) & \textit{0.0679} (798.64) \\
        3 & 0.0901 (2382.6) & \textbf{0.0385} (2372.3) & 0.5747 (750.75) & \textit{0.0582} (801.10) \\
        \bottomrule
    \end{tabular}
\end{table}

\begin{table}
\centering
\caption{The bias and variance results of the estimators derived from the TL1 and nuclear norm under Scheme 2 with SNR = 10 over 100 random trials}
\label{tab_d2}
\begin{minipage}{0.4\textwidth}
\centering

\begin{tabular}{cccc}
\hline
\multicolumn{4}{c}{TL1} \\
\hline
$m_1 = m_2$ & (r, SR) & $\text{bias}^2$ & variance \\
\hline
300 & (5, 0.1) & 0.1666 & 0.0613 \\
 & (5, 0.2) & 0.0028 & 0.0269 \\
\hline
\end{tabular}
\end{minipage}%
\hspace{1pt}
\begin{minipage}{0.4\textwidth}
\centering

\begin{tabular}{cccc}
\hline
\multicolumn{4}{c}{Nulcear norm} \\
\hline
$m_1 = m_2$ & (r, SR) & $\text{bias}^2$ & variance \\
\hline
300 & (5, 0.1) & 3.0098 & 0.0310 \\
 & (5, 0.2) & 1.9715 & 0.0182 \\
\hline
\end{tabular}
\end{minipage}
\end{table}

\section{Real dataset description}
\label{A7.real-data}
We conduct experiments on two real datasets:
\begin{enumerate}
    \item Coat Shopping Dataset\footnote{\url{https://www.cs.cornell.edu/~schnabts/mnar/}} contains ratings contributed by 290 Turkers for a comprehensive inventory of 300 items. The training set comprises 6960 non-uniform, self-selected ratings, while the test set includes 4640 uniformly selected ratings. A more detailed description is provided in \cite{schnabel2016recommendations}.
    \item Movielens 100K Dataset\footnote{\url{https://www.kaggle.com/datasets/prajitdatta/movielens-100k-dataset}} is collected by the GroupLens Research Project at the University of Minnesota \citep{harper2015movielens}. This dataset contains 100,000 ratings from 943 users across 1682 movies and is organized into ten subsets, comprising five distinct training sets, and their corresponding test sets.
\end{enumerate}

\end{document}